\pdfoutput=1
\documentclass[11pt]{article}

\usepackage[margin=1in]{geometry}
\usepackage{amsthm,amsmath,amsfonts,amssymb}
\usepackage[numbers]{natbib}
\usepackage{graphicx}
\usepackage{xkeyval}
\usepackage{enumitem}
\usepackage{mathrsfs}
\usepackage{mathtools}
\usepackage{dsfont}
\usepackage{subcaption}
\usepackage{stackengine}
\usepackage{appendix}
\usepackage[most]{tcolorbox}
\usepackage{changepage}
\usepackage{mdframed}
\usepackage{float}
\usepackage{tikz-3dplot}
\usepackage{url}
\usepackage{booktabs}
\usepackage{array}
\usepackage[colorlinks,citecolor=blue,urlcolor=blue,linkcolor=blue]{hyperref}
\usepackage{orcidlink}

\numberwithin{equation}{section}

\DeclareMathOperator*{\argmin}{\arg\,\min}

\DeclareMathOperator{\atantwo}{atan2}
\newcommand{\defeq}{\vcentcolon=}

\newcommand{\E}[2]{\mathbb{E}_{#1}\left[#2\right]}
\newcommand{\Prob}[2]{\mathbb{P}_{#1}\left(#2\right)}

\newcommand{\lnr}[1]{\ln\left( #1\right)}
\newcommand{\expr}[1]{\exp\left( #1\right)}
\newcommand{\Ent}[2]{\text{H}_{#1}\left[ #2\right]}

\RenewCommandCopy{\leq}{\leqslant}
\RenewCommandCopy{\geq}{\geqslant}
\renewcommand{\d}[1]{\operatorname{d}\!{#1}}

\theoremstyle{plain}
\newtheorem{theo}{Theorem}
\newtheorem{prop}[theo]{Proposition}
\newtheorem{lem}[theo]{Lemma}
\newtheorem{cor}[theo]{Corollary}

\theoremstyle{definition}

\newtheorem{rem}[theo]{Remark}
\newtheorem{defi}[theo]{Definition}

\title{A resolution of the Borel--Kolmogorov Paradox via the Maximum Entropy Principle}

\author{%
Rapha\"el Tr\'esor\,\orcidlink{0009-0009-0411-9002}
\qquad
Mykola Lukashchuk\,\orcidlink{0009-0003-3011-0825}
\\[1ex]
\small Department of Electrical Engineering, Eindhoven University of Technology\\
\small 5600 MB Eindhoven, The Netherlands\\[0.5ex]
\small \href{mailto:r.v.tresor@tue.nl}{\texttt{r.v.tresor@tue.nl}}, \href{mailto:m.lukashchuk@tue.nl}{\texttt{m.lukashchuk@tue.nl}}
}

\date{}

\begin{document}

\maketitle

\begin{abstract}
Bayesian updating routinely conditions on events of prior probability zero, such as an exact observation of a continuous parameter or a parameter confined to a submanifold.
Two equally natural parametrizations of the same event can then return different posteriors, which is the Borel--Kolmogorov paradox.
We add a metric to the probabilistic model and define the posterior given a closed null set as the limit of Maximum Entropy solutions under a vanishing constraint on the distance to that set.
We give a sufficient condition for existence and prove that, given this metric, the posterior is unique and invariant under measure-preserving isometries.
It recovers the textbook Bayes formulas where these apply unambiguously, on positive-measure sets and in Euclidean models.
Each stage of the limit has an explicit density that depends only on the distance to the conditioning set, so the posterior can be computed with standard tools for Kullback--Leibler minimization.
The construction does not remove the choice of a metric but makes it an explicit modelling input.
In a controlled experiment where the geometry of the measurement is known, we show that a Bayesian analysis ignoring this geometry can yield uncalibrated posteriors.
On real palaeomagnetic data, where the geometry is disputed, we show how candidate metrics can be compared within the same Bayesian framework.
\end{abstract}

\noindent\textbf{Keywords:} Bayesian inference; Bayesian inference on manifolds; Borel--Kolmogorov paradox; Conditional probability; Directional statistics; Disintegration; Kullback--Leibler divergence; Maximum Entropy Principle; Regular conditional distribution; Riemannian geometry.

\smallskip
\noindent\textbf{MSC2020:} 62F15 (primary); 62H11, 62B10, 94A17, 60A10, 53B20 (secondary).

\section{Introduction}\label{sec:intro}
Conditional probability of an event \(B\) given an event \(A\) is classically defined as
\begin{equation}\label{def: conditional probability}
 \Prob{}{B\mid A} \defeq \frac{\Prob{}{B \cap A}}{\Prob{}{ A}} \, ,
\end{equation}
whenever \(\Prob{}{A}>0\) \citep[Ch.~1.4]{kolmogorov_foundations_2018}.

This underlies Bayes' rule and prior updating, but it fails exactly when the conditioning event is null under the prior, a case common enough in Bayesian practice that textbooks provide a separate update formula for it.
To extract information about a parameter \(x\) from an observed value \(\hat{y}\), introduce two random variables \((X,Y)\in\mathbb{R}^p\times\mathbb{R}^q\) with joint prior density \(p(x,y)\) with respect to the Lebesgue measure \(\mathrm{d}x\,\mathrm{d}y\), and use the normalized ratio of densities
\begin{equation}\label{eq: Baye's formula for lebesgue}
X\mid Y=\hat{y} \sim p(x\mid \hat y)\, \mathrm{d}x \defeq \frac{p(x,\hat y)}{\int_{\mathbb{R}^p} p(x,\hat y)\,\mathrm{d}x}\, \mathrm{d}x \,,
\end{equation}
whenever the denominator is finite and strictly positive.

This extension of Bayes' rule to null events is ambiguous.\footnote{The example that follows is taken from \citet{michael_proschan_expect_1998}.}
Let \(X_1\sim\mathcal{N}(0,1)\) and \(X_2\sim\mathcal{N}(0,1)\) be independent.
What is the law of \(X_1\) given the null event \(A=\{X_1=X_2\}\)?
Applying \eqref{eq: Baye's formula for lebesgue} requires an auxiliary random variable \(Y\) whose value at some point encodes \(A\); the practitioner forms the joint law of \((X_1,Y)\) and reads off the posterior.

Taking \(Y=X_1-X_2\) and conditioning on \(Y=0\) gives
\begin{equation}\label{eq: intro posterior difference}
p(x_1 \mid Y=0) = \pi^{-1/2}\expr{-x_1^2}\,,
\end{equation}
whereas taking \(Y=X_1/X_2\) and conditioning on \(Y=1\) gives
\begin{equation}\label{eq: intro posterior ratio}
p(x_1 \mid Y=1) = |x_1|\,\expr{-x_1^2}\,.
\end{equation}
Both encode the same prior \(\mathbb{P}\) and the same event \(A\), yet the posteriors disagree: the first is maximal at \(x_1=0\), the second vanishes there.
Nothing in the textbook procedure says which to prefer.
This ambiguity is the Borel--Kolmogorov paradox.

Conditioning on ambiguous null set is a recurring situation: when the observation is an exact value of a continuous quantity (for hypothesis testing \cite{marin_savage_paradox_2010}, or probabilistic programming \cite{jacobs_paradox_programm_2021}), when the parameter space is itself a manifold (a unit-norm direction on the sphere \cite{mardia_jupp_directional}, an orientation in the group of 3D rotations \cite{baran2025,barfoot}, or simply a submodel \cite{consini_prior_specification_2008}), or when a physical law confines it to an equilibrium or conservation manifold \cite{meehan_borel-kolmogorov_2021, xu_right_2026}.
In these cases, applying equation \eqref{eq: Baye's formula for lebesgue} on equally reasonable parametrizations returns different posteriors, hence multiple posterior means, multiple credible regions; posing a problem for Bayesian inference practice  \citep{rescorla_epistemological_2015, arnold_conditional_2003, bordley_avoiding_2015}.

Conditioning on non-null events became standard statistical practice due to \eqref{def: conditional probability} being mathematically founded, computable, and transparent (the posterior follows from the pair \((\mathbb{P},A)\) alone).
Conditioning on null events extend this practice, thus should held to the same requirements.
Combining the Maximum Entropy Principle (MaxEnt) \citep{wu_maximum_1997,jaynes_information_1957} with a metric \(d\) on the parameter space, this work proposes a resolution meeting those requirements:

\begin{enumerate}
    \item A framework yielding a unique conditional probability for any triple \((\mathbb{P},d,A)\), including \(\Prob{}{A}=0\), with existence, uniqueness and invariance under measure-preserving isometries, and recovering the textbook formulas wherever conditioning is unambiguous (Section~\ref{section: Bayes rule for null measure event}).
    \item Algorithms to approximate this conditional probability when no closed form is available (Section~\ref{sec:computation}).
    \item A methodology for Bayesian inference on null events in which the ambiguity is removed from the data analysis by adding a clear geometrical choice to the initial model, separating cleanly the indeterminacy underlying that modelling choice from an unambiguous Bayesian data analysis based on it (Section~\ref{sec:methodology}).
\end{enumerate}

We test our methodology on palaeomagnetic data, where the Borel--Kolmogorov paradox arises and the choice of a metric for the statistical analysis is already debated in the literature.
We show that candidate metrics can be evaluated from data within our Bayesian framework (Section~\ref{sec:real-data-example}).
Finally, we compare our definition with the existing literature on conditional probability on null events, with regard to Bayesian inference practice (Sections~\ref{sec:related-work} and~\ref{subsec:method-vs-others}), and discuss how our work reconciles competing arguments on the choice of a posterior (Section~\ref{sec: discussion}).

\section{Related Work and Background}\label{sec:related-work}

\subsection{Conditioning and the Borel-Kolmogorov Paradox}\label{subsec:condi-BK-paradox}

\citet[Sec.~39]{borel_ements_1909} and \citet[Sec.~5.2]{kolmogorov_foundations_2018} observed that Kolmogorov's axioms alone do not extend conditional probability \eqref{def: conditional probability} uniquely to null events.
The most common approach, \emph{conditional expectation} \cite{klenke_conditional_2020}, \cite[Th.~8.1]{kallenberg1997foundations}, conditions on a sub-\(\sigma\)-algebra \(\mathcal{G}\) rather than on an event.
On a standard Borel space, the \emph{disintegration theorem} \cite[Th.~8.5]{kallenberg1997foundations} provides a \emph{regular conditional probability} \cite{faden-condi-proba}, a version of the conditional probabilities that, for each outcome, is a probability measure on the whole \(\sigma\)-algebra.
Both are defined only up to \(\mathbb{P}\)-null events, on which they can be changed arbitrarily, even once \(\mathcal{G}\) is fixed \cite[Sec.~3, eqs.~(44)--(48)]{gyenis_conditioning_2017}.

The \emph{strategies} \cite{berti_notion_2020,dubins-disintegrations-1975}, rooted in de Finetti's \emph{coherent conditional probabilities}, fix a \emph{proper} conditional probability on every atom of \(\mathcal{G}\), including the \(\mathbb{P}\)-null ones.
Since many strategies exist for the same \((\mathbb{P},\mathcal{G})\), they characterize the admissible answers without selecting one.

The other approach defines conditioning from events rather than from a sub-\(\sigma\)-algebra, either axiomatically as a primitive object \cite{Renyi-1955} or from \eqref{def: conditional probability} \cite[Sec.~15.7]{jaynes_probability_2003}.
In both cases, conditioning on a null event is defined as a limit of conditional probabilities given enlarged sets of positive probability, and different approximating sequences for the same event $A$ yield different limits. 

\citet[Sec.~2.5]{mosegaard2002probabilistic} extend this limit to an arbitrary metric $d$, conditioning on $\{x : d(x,A)\le\varepsilon\}$ with $\varepsilon\to 0$, and present it as an unambiguous definition.

Endowing the spaces with a topology, \citet{tjur-constructive-1975} defines the conditional probability given $\{t(X) = \hat y\}$ as the weak limit of \eqref{def: conditional probability} along the net of sets $\{t^{-1}(B)\}$, $B$ shrinking to $\hat y$. 
For a fixed map $t$, the limit is unique with Tjur's property, and renders a particular choice of regular conditional probability \citep[Lemma~2.7]{ACKERMAN_FREER_ROY_2017}.
The Borel--Kolmogorov ambiguity lies instead in the choice of $t$: with the Euclidean topology, $t(x) = x_1 - x_2$ at $0$ and $t(x) = x_1/x_2$ at $1$ both have Tjur's property, yet return \eqref{eq: intro posterior difference} and \eqref{eq: intro posterior ratio}.

The multiplicity of rigorous conditional probability measures on null events is a consequence of their measure-theoretic indistinguishability from the empty event.
Probability theory on its own produces an ambiguous Bayesian inference method on \(\mathbb{P}\)-null events, making the paradox epistemological rather than purely mathematical \cite{meehan_borel-kolmogorov_2021,rescorla_epistemological_2015}.
A discriminating criterion must be added to conditioning theory, and the resulting method must be judged as a statistical method, beyond plain mathematical rigour.
None of the constructions above says how the choice, of a strategy, an approximating sequence, a metric or a map $t$, should be made for Bayesian inference.

We make this choice a modelling one, a metric declared together with the prior and testable on data (Section~\ref{sec:methodology}), and derive the resulting posterior rigorously from MaxEnt, which makes it computable (Sections~\ref{section: Bayes rule for null measure event} and~\ref{sec:computation}).

\subsection{Maximum Entropy Principle}\label{subsec:MAxEnt}

Both Bayes’ rule and MaxEnt are \emph{update procedures}: given a prior probability measure and information encoded by constraints, they return a posterior measure. In MaxEnt, the posterior minimizes the Kullback--Leibler (KL) divergence subject to the constraints.

\begin{defi}[Kullback--Leibler divergence / Relative entropy]\label{def: rela-ent}
Let $(E,\mathcal{E})$ be a measurable space and let $\nu$ be a $\sigma$-finite measure. For a probability measure $\mu$ with $\mu\ll\nu$,
\begin{equation}\label{eq: rela-ent}
    \Ent{\nu}{\mu} \coloneqq \int_E \frac{\mathrm{d}\mu}{\mathrm{d}\nu}(x)\,
    \lnr{\frac{\mathrm{d}\mu}{\mathrm{d}\nu}(x)}\,\mathrm{d}\nu(x).
\end{equation}
If $\mu\not\ll\nu$, set
\begin{equation}\label{eq: rela-ent-not-continuous}
    \Ent{\nu}{\mu} \coloneqq +\infty.
\end{equation}
\end{defi}

\begin{prop}[\cite{cover_elements_2001}, Thm.~2.6.3]\label{prop: rel ent is positive}
Let $\mu,\nu$ be probability measures on $(E,\mathcal{E})$ with $\mu\ll\nu$. Then $\Ent{\nu}{\mu}\ge 0$, with equality iff $\mu=\nu$.
\end{prop}

\begin{defi}[MaxEnt Problem]\label{def:MaxEnt}
Let $\nu$ be a prior probability measure on $(E,\mathcal{E})$. 
Given a feasible set \(I \subset \mathcal{M}(E)\) the MaxEnt problem $(\nu,I)$ is \(\inf_{\mu\in I}\; \Ent{\nu}{\mu}\) and it is called \emph{well-defined} if the infimum is attained by a unique $\mu^\star$. 
\end{defi}

MaxEnt has proved effective in statistical mechanics \citep{jaynes_information_1957}, image reconstruction \citep{gull_maximum_1984}, ecology \citep{baldwin_use_2009} and Bayesian modelling \citep{clarke_information_2007}.

Moreover, when $\mathbb{P}(A)>0$, equation~\eqref{def: conditional probability} arises from MaxEnt from the constraint
\( I=\big\{\mathbb{Q}:\ \mathbb{Q}(A)=1\big\} = \Big\{ \mathbb{Q}:\ \int_E \mathds{1}_A\,\mathrm{d}\mathbb{Q}=1\Big\}\);
see \cite{williams_bayesian_1980}, or Appendix~\ref{sec: conditional proba}.
Thus, where conditioning is unambiguous, Bayes’ rule is a special case of MaxEnt. 
This recovery is an elementary case of the variational machinery (relative-entropy minimization, exponential tilting, concentration \cite{csiszar_Idivergence_1975,chandrasekaran_relative_nodate}). 

The link between conditioning and MaxEnt is classical, as it is a direct application of convex-duality \cite{borwein_duality_1991, borwein_partially_1992}.
The derivation of MaxEnt from conditioning is done in \cite{skyrms_maximum_1985,van_Campenhout_MaxEnt,csiszar_GeneralizedIProjection, dembo_refinements_1996}\cite[\S7.3]{dembo_applications_2010_chap7}.
However, our derivation of conditional probability on null sets from MaxEnt is new.

To be computed, the formula of \citet{bungert_lion_2022} requires a parametrization of the conditioning event satisfying a regularity condition \citep[Th.~1]{bungert_lion_2022}, to be found case by case.
The limits of Tjur and of \citet{mosegaard2002probabilistic} are approximated by applying \eqref{def: conditional probability} to enlarged sets around the conditioning event, at an accuracy that degrades with the dimension \citep{barber2015abc}, or through kernel relaxations sampled by sequential Monte Carlo \citep{cockayne2019bpnm}.
Each stage of the limit defining our posterior is instead a density positive on all of $E$, which requires only the distance to the conditioning set, neither a parametrization of it nor a conditioning map.
It can therefore be sampled with standard gradient-based Markov chain Monte Carlo (MCMC) and, being the solution of a KL minimization, handled by the large and still growing literature on this problem (Section~\ref{sec:computation}).
Table~\ref{tab:comparison} compares the resulting conditioning with the constructions of Section~\ref{subsec:condi-BK-paradox}.

\section{Bayes' Rule from MaxEnt}\label{section: Bayes rule for null measure event}

\subsection{MaxEnt Posterior}

We work on a standard Borel space $(E,d,\mathcal{B}(E))$ endowed with a prior probability measure $\nu$. Let $A\subset\operatorname{supp}(\nu)$ be closed. Write $d(x,A)\coloneqq\inf_{a\in A} d(x,a)$, set $d_R(x,A)\coloneqq\min\{d(x,A),R\}$ for $R>0$, and  
\begin{equation}\label{eq:neighboor-extension-def}
    A^\eta \coloneqq \left\{ x \, |\, d(x,A) \leq \eta \right\} \,.
\end{equation}

The natural way to force a posterior onto $A$ is the hard constraint $\int_E d(x,A)^2\,\mathrm{d}\mu(x)=0$, which for closed $A$ holds exactly when $\operatorname{supp}(\mu)\subset A$: the integrand is continuous and nonnegative, so it vanishes $\mu$-a.e., and outer regularity extends this to every Borel subset of $A^{\complement}$.
That constraint is unusable on a null set.
Any solution of a well-defined MaxEnt problem (Definition~\ref{def:MaxEnt}) is absolutely continuous with respect to $\nu$, since $\mathrm{KL}(\mu\|\nu)=\infty$ whenever $\mu\not\ll\nu$; when $\nu(A)=0$ no such measure can be supported on $A$, and the feasible set is empty.
The resolution is to relax the constraint and pass to the limit: each member of the sequence solves a genuine MaxEnt problem, whose support is not confined to $A$, while support inclusion is asked only of the limit, where no MaxEnt problem has to be solved.
The distance is truncated at $R$ to make the constraint function bounded, as the duality theorem behind Theorem~\ref{theo:general bayes rule has a coherent solution} requires; the level $R$ does not affect the limit (Proposition~\ref{prop:general bayes rule indepedant of R clamping}).

\begin{lem}[Support inclusion from truncated constraints]\label{lem:distance contraint equiv to support inclusion even for limit}
Let $(E,d)$ be standard Borel and $A\subset E$ closed. If $\mu_n\Rightarrow\mu$, where $\Rightarrow$ denotes weak convergence, and, for every $R>0$,
\begin{equation}\label{eq:relax distance contraint equiv to support inclusion}
\lim_{n\to\infty}\int_E d_R(x,A)^2\,\mathrm{d}\mu_n(x)=0,
\end{equation}
then $\operatorname{supp}(\mu)\subset A$.
\end{lem}

\begin{proof}
See Appendix~\ref{subsec: proof of lemma Support inclusion from truncated constraints}.
\end{proof}

\begin{defi}[MaxEnt Posterior]\label{def:general-bayes-rule}
Fix $R>0$. Given $(E,d,\mathcal{B}(E),\nu)$ and a closed $A\subset\operatorname{supp}(\nu)$, consider for $\sigma>0$ the MaxEnt problem
\begin{equation}\label{prob: Bayes for Hausdorff space}
\inf_{\mu}\Big\{\mathrm{KL}(\mu\|\nu):\ \mu(E)=1,\ \int_E d_R(x,A)^2\,\mathrm{d}\mu(x)\le \sigma^2\Big\}.
\end{equation}
Assume~\eqref{prob: Bayes for Hausdorff space} is well-defined with solution $\mu_\sigma$ and that $\mu_\sigma\Rightarrow \mu^\star$ as $\sigma\downarrow 0$. We call $\mu^\star$ the MaxEnt posterior for $R$.
\end{defi}

In the following of the section, we justify our definition by proving some of its properties.

\begin{theo}[Pre-Posterior]\label{theo:general bayes rule has a coherent solution}
Let $(E,d)$ be standard Borel with prior $\nu$ and $A\subset\operatorname{supp}(\nu)$ closed. For each $\sigma>0$, problem~\eqref{prob: Bayes for Hausdorff space} admits a unique solution of the form
\begin{equation}\label{def: pre Bayes posterior}
\mathrm{d}\mu_{a(\sigma)}(x)=\frac{\exp\{-a(\sigma)\, d_R(x,A)^2\}}{\int_E \exp\{-a(\sigma)\, d_R(u,A)^2\}\,\mathrm{d}\nu(u)}\,\mathrm{d}\nu(x),
\end{equation}
where $a(\sigma)\ge 0$ is chosen to satisfy the constraint. If $\nu(E\setminus A)>0$ and $\sigma\downarrow 0$, then $a(\sigma)\uparrow\infty$.
\end{theo}

\begin{proof}See Appendix~\ref{subsec: proof of theorem Pre-Posterior}.\end{proof}

The truncation does not affect the limit:

\begin{prop}[Independence of truncation]\label{prop:general bayes rule indepedant of R clamping}
Let $R,R'>0$ and suppose Definition~\ref{def:general-bayes-rule} yields a MaxEnt posterior $\mu^\star$ for $R$. Then the posterior for $R'$ exists and equals $\mu^\star$.
\end{prop}

\begin{proof}See Appendix~\ref{subsec: proof of proposition Independence of truncation}.\end{proof}

The MaxEnt posterior therefore does not depend on $R$, and we denote it $\nu(\cdot\mid A)$.
Since $\int_E d_R(x,A)^2\,\mathrm{d}\mu_\sigma(x)\le\sigma^2\to 0$, Lemma~\ref{lem:distance contraint equiv to support inclusion even for limit} gives $\operatorname{supp}(\nu(\cdot\mid A))\subset A$, i.e. $\nu(A^{\complement}\mid A)=0$.

\begin{prop}[Invariance under measure-preserving isometries]\label{prop:invariance-by-homeomorphism}
Let $h:(E,d)\to(\widetilde{E},\widetilde{d})$ be an isometry and let $h_{\#}\nu$ denote the push-forward measure. If the $(\nu,h^{-1}(A))$ MaxEnt posterior is well-defined on $(E,d)$, then the $(h_{\#}\nu,A)$ posterior is well-defined on $(\widetilde{E},\widetilde{d})$ and
\[
h_{\#}\nu(B\mid A)=\nu\big(h^{-1}(B)\mid h^{-1}(A)\big),\qquad B\in\mathcal{B}(\widetilde{E}).
\]
\end{prop}

\begin{proof}See Appendix~\ref{subsec: proof of proposition Invariance under measure-preserving isometries}.\end{proof}

The invariance holds under isometries, not under homeomorphisms: on $\{x_1x_2\neq 0\}$, the Euclidean metric and its pull-back by $(x_1,x_2)\mapsto(x_1,x_1/x_2)$ induce the same topology but give two different posteriors under Definition~\ref{def:general-bayes-rule}, see Section~\ref{subsec:method-vs-others}.

We next give a sufficient condition for existence, together with a formula for the MaxEnt posterior.

\begin{theo}[Existence criterion]\label{theo: criteria of existence}
Let $A\subset E$ be closed with $\nu(A)=0$, and let $\eta_0>0$ and $\pi:\{x\in E:\ d(x,A)\le\eta_0\}\to A$ be a Borel map with $d(x,\pi(x))=d(x,A)$.
For Borel $C\subset A$ and $0\le\eta\le\eta_0$, set
\begin{equation*}
C^\eta\coloneqq\{x\in E:\ d(x,A)\le \eta,\ \pi(x)\in C\}\,.
\end{equation*}
Assume that each map $\eta\mapsto \nu(C^\eta)$ (for $C\subset A$ Borel) is absolutely continuous on $[0,\eta_0]$ with density $f_C$, that $f_C$ and $f_A$ are continuous at $0$, and that $f_A(0)>0$.
Then the MaxEnt posterior exists and
\begin{equation}\label{eq:existence-criterion-posterior}
\nu(C\mid A)=\frac{f_C(0)}{f_A(0)}\,,\qquad C\subset A \;\text{ Borel}\,.
\end{equation}
\end{theo}

\begin{proof}See Appendix~\ref{subsec: proof of theorem criteria of existence}.\end{proof}

A map $\pi$ as above exists whenever closed bounded subsets of $E$ are compact (e.g. $\mathbb{R}^n$, $\mathbb{S}^2$ or $\mathrm{SO}(3)$), by a measurable selection of nearest points.

\subsection{Recovering classical formulas}\label{subsec:recovering-classical-formulas}
 
Before turning to the geometry of null sets, we check that the construction of Definition~\ref{def:general-bayes-rule} reproduces what is already established: the classical conditional probability when the conditioning event has positive measure, and the continuous Bayes formula in Euclidean models when a random variable is observed exactly.
 
\begin{theo}\label{theo: Bayesian map is an extension of conditional probability}
Let $(E,d,\mathcal{B}(E),\nu)$ be standard Borel and $A\in\mathcal{B}(E)$ with $\nu(A)>0$. Let $\mu_n$ solve~\eqref{prob: Bayes for Hausdorff space} with $\sigma_n\downarrow 0$. Then
\(\mu_n \Rightarrow \frac{\mathds{1}_{A}}{\nu(A)}\,\nu\),
i.e., the MaxEnt posterior equals the normalized restriction of $\nu$ to $A$.
\end{theo}

\begin{proof}See Appendix~\ref{subsec: proof of theorem bayesian map is an extension of conditional probability}.\end{proof}

\begin{prop}[Concentration of the $Y$-marginal]\label{lem: convergence criteria to delta measure Polish space}
Let $(X,Y)\in(\mathcal{X}\times\mathcal{Y},d)$ be random with prior $\nu$ and suppose $d$ dominates a metric $d_{\mathcal{Y}}$:
\(d_{\mathcal{Y}}(y_1,y_2)\le d\big((x,y_1),(x,y_2)\big)\) for all \(x\in\mathcal{X}\).
Set $A=\mathcal{X}\times\{\hat y\}$ with $\hat y\in \mathrm{int}(\operatorname{supp}(\nu))$.
Let $\mu_n$ solve~\eqref{prob: Bayes for Hausdorff space} with $\sigma_n\downarrow 0$. Then
\(\mu_n(\mathcal{X}\times\cdot)\Rightarrow \delta_{\hat y}\).
\end{prop}

\begin{proof}See Appendix~\ref{subsec: proof of lemma convergence criteria to delta measure Polish space}.\end{proof}

Under mild regularity, the $X$-marginal admits a closed form.

\begin{theo}[Bayes rule for Product Space]\label{theo: Bayes rule for product space}
Let $(E,d_E)$ and $(F,d_F)$ be $\sigma$-compact standard Borel spaces and endow $E\times F$ with a product metric $d$ satisfying
\begin{subequations}\label{eq:condition on the metric}
\begin{align}
d\big((x_1,y),(x_2,y)\big)&= d_E(x_1,x_2),\\
d\big((x,y_1),(x,y_2)\big)&= d_F(y_1,y_2).
\end{align}
\end{subequations}
Let $\mathrm{d}\nu(x,y)=p(x,y)\,\mathrm{d}\mathcal{N}_E(x)\,\mathrm{d}\mathcal{N}_F(y)$ with Radon $\mathcal{N}_E,\mathcal{N}_F$ and continuous $p\ge 0$, and set $\widetilde p(y)\coloneqq\int_E p(x,y)\,\mathrm{d}\mathcal{N}_E(x)$.
If $\hat y\in\mathrm{int}(\operatorname{supp}(\nu))$, $\widetilde p$ is continuous at $\hat y$ and $0<\widetilde p(\hat y)<\infty$, then
\begin{equation}\label{eq: bayes posterior for product space}
\mathrm{d}\nu\big(x,y \mid E\times\{\hat y\}\big)
= \frac{p(x,\hat y)}{\int_{E}p(x,\hat y)\,\mathrm{d}\mathcal{N}_E(x)}\,\mathrm{d}\mathcal{N}_E(x)\otimes \delta_{\hat y}(y).
\end{equation}
\end{theo}

\begin{proof}See Appendix~\ref{subsec: proof of theorem Bayes rule for product space}.\end{proof}

The conditions on $\widetilde p$ state that the evidence is finite, positive and continuous at the observation $\hat y$, which is exactly when the textbook continuous Bayes update is meaningful there.
On Euclidean models, Definition~\ref{def:general-bayes-rule} therefore recovers this update.

\begin{cor}[Euclidean/Lebesgue case]\label{theo: Bayes rule for Lesbegue continuous measures}
Let $(X,Y)\in\mathbb{R}^p\times\mathbb{R}^q$ with Euclidean metric and prior $\nu=p(x,y)\,\mathrm{d}\mathscr{L}^p(x)\,\mathrm{d}\mathscr{L}^q(y)$, where $p$ is continuous.
If $\hat y\in\operatorname{supp}(\nu)$, $y\mapsto\int_{\mathbb{R}^p}p(x,y)\,\mathrm{d}\mathscr{L}^p(x)$ is continuous at $\hat y$ and $0<\int_{\mathbb{R}^p}p(x,\hat y)\,\mathrm{d}\mathscr{L}^p(x)<\infty$, then
\begin{equation}\label{eq: bayes posterior for lebesgue continuous random variables}
\mathrm{d}\nu\big(x,y \mid \mathbb{R}^p\times\{\hat y\}\big)
= \frac{p(x,\hat y)}{\int_{\mathbb{R}^p} p(x,\hat y)\,\mathrm{d}\mathscr{L}^{p}(x)}\,\mathrm{d}\mathscr{L}^{p}(x)\otimes \delta_{\hat y}(y),
\end{equation}
i.e. equation~\eqref{eq: Baye's formula for lebesgue}.
\end{cor}

\begin{rem}
In Corollary~\ref{theo: Bayes rule for Lesbegue continuous measures}, Lebesgue continuity implies $\nu(\operatorname{supp}(\nu)\setminus \mathrm{int}(\operatorname{supp}(\nu)))=0$, so $\hat y\in \mathrm{int}(\operatorname{supp}(\nu))$ is a mild assumption.
\end{rem}

In summary, Definition~\ref{def:general-bayes-rule} produces a well-defined posterior that (i) is supported on the conditioning set, (ii) is invariant under measure-preserving isometries, (iii) coincides with classical conditional probability on positive-measure sets, and (iv) recovers the textbook continuous Bayes update in Euclidean models (the Euclidean/Lebesgue Corollary~\ref{theo: Bayes rule for Lesbegue continuous measures}). Section~\ref{sec: discussion} draws out the practical reading of properties (iii)--(iv).

\begin{rem}[Relation to the coarea formula]\label{rem:coarea}
When $A$ is a rectifiable submanifold of a Riemannian manifold, the posterior of Definition~\ref{def:general-bayes-rule} coincides with the classical induced-volume (coarea) conditional \cite{federer_geometric_1996, pennec_intrinsic_2006, bhattacharya_nonparametric_2003}, whereas Theorem~\ref{theo: criteria of existence} also applies to closed sets $A$ that carry no induced volume.
\end{rem}

\section{Computing the MaxEnt Posterior}\label{sec:computation}

Theorem~\ref{theo:general bayes rule has a coherent solution} solves \eqref{prob: Bayes for Hausdorff space} in closed form, up to a normalizing constant over $E$; whenever the limit itself has a closed form, the constant is an integral over $A$.
Either integral is a standard quadrature problem \citep{davis_methods_1984}, whose cost grows quickly with the dimension.
MCMC requires the density only up to a constant and runs for $E$ Euclidean \citep{jones_markov_2022} as well as for $E$ a manifold \citep{byrne_geodesic_2013, diaconis_sampling_2013}.
When the limit has no closed form, the pre-posterior can be sampled at large $a$ with gradient-based MCMC \citep[e.g.,][]{hoffman2014nuts}, as done by \citet{presman_distance_2023} for the Euclidean metric.
It applies to any metric for which $d(x,A)$ and its gradient can be evaluated.

KL minimization methods include iterative scaling or Bregman projections on a grid \citep{darroch_generalized_1972, benamou_iterative_2015, csiszar_GeneralizedIProjection} (limited to low dimension); variational inference over a parametric family by stochastic or natural gradient descent \citep{blei_variational_2017, hoffman_stochastic_2013}; and particles transported along a gradient flow of the objective, in a reproducing kernel Hilbert space \citep{liu_stein_2016} or for the Fisher--Rao and Wasserstein metrics \citep{wang_information_2020}.

Any finite $a$ only approximates the posterior; how large $a$ must be for a prescribed accuracy is left open.
By Proposition~\ref{prop:general bayes rule indepedant of R clamping}, $R$ may be set for numerical convenience.
Evaluating $d(x,A)$ is required in every case and is not discussed here.

\section{Bayesian Inference with a Metric}\label{sec:methodology}

Running on the historical example of the Borel--Kolmogorov paradox \cite[Sec.~43]{borel_ements_1909}, we show the MaxEnt posteriors (Definition~\ref{def:general-bayes-rule}) gives an intuitive understanding of the multiplicity of posteriors.
Then, we show that there is no a universally correct formula, but only the formula derived from the data-generating metric give calibrated posterior.
From these observations, we propose a practical methodology for Bayesian inference.

\subsection{Geometric interpretation of the multiplicity}\label{subsec:geometrical-resolution}

An object lies somewhere on the surface of the Earth and its position is to be estimated.
Nothing is known at first, so the prior is uniform over the globe.
A measurement then reveals the meridian on which the object lies, and the question is where along that meridian it is.

Let $(\mathbb{S}^2,g)$ be the $2$-sphere with its Riemannian metric $g$ induced from $(\mathbb{R}^3,d_{\mathrm{euclid}})$. 
Denote by $d_g$ the geodesic distance. Equip $(\mathbb{S}^2,d_g)$ with its Borel $\sigma$-algebra and the uniform prior probability measure

\begin{equation}\label{eq:sphere-uniform-measure}
    \nu(B)=\frac{\mathrm{Area}(B)}{\mathrm{Area}(\mathbb{S}^2)},\qquad B\in\mathcal{B}(\mathbb{S}^2),
\end{equation}

where the area is computed via the volume form associated with $g$ \cite[Chap.~16]{lee_introduction_2012},
    \(\mathrm{Area}(B)\defeq \int_{B} \sqrt{\det g(a)}\,\mathrm{d}a\),
which is invariant under diffeomorphisms \cite[App.~B]{lee_introduction_2018}.

The meridian has prior probability zero, so the estimate calls for the update rule of equation~\eqref{eq: Baye's formula for lebesgue} or Definition~\ref{def:general-bayes-rule}.
On the chart 

\begin{equation}\label{eq:chart on sphere}
  \mathbb{S}^2\ni(x,y,z)\mapsto (\varphi,\theta)=\Big(\arcsin(z),\ \atantwo(y,x)\Big)\in\bigl(-\tfrac{\pi}{2},\tfrac{\pi}{2}\bigr)\times(-\pi,\pi],
\end{equation}
with $\atantwo(y,x)=2\arctan\!\left(\dfrac{y}{\sqrt{x^2+y^2}+x}\right)$, the marginalization rule reports a posterior proportional to $\cos\varphi$, but a flat posterior once the coordinates are changed by a diffeomorphism that exchanges the roles of $\varphi$ and $\theta$: the Borel--Kolmogorov paradox \cite[Sec.~5.2]{kolmogorov_foundations_2018}.
Both posteriors are also obtained under Definition~\ref{def:general-bayes-rule}.

\begin{prop}[Posteriors on great circles under different metrics]\label{theo: Bayesian posterior on the great circle with geodesic distance}
Let $(\mathbb{S}^2,d_g)$ be measured by the uniform prior \(\nu\) of equation \eqref{eq: uniform measure on sphere}. 
In these coordinates, $\nu$ is 
\begin{equation}\label{eq: uniform measure on sphere}
    \nu(\mathrm{d}\varphi,\mathrm{d}\theta)=\frac{\cos\varphi}{4\pi}\,\mathrm{d}\varphi\,\mathrm{d}\theta.
\end{equation}
Then the MaxEnt posterior of any great circle is uniform. 
In the chart \eqref{eq:chart on sphere},
\begin{equation}\label{eq: uniform posterior on meridian}
    \underbrace{\frac{1}{2\pi}\,\mathrm{d}\varphi \,\times\, \big(\delta_{0}(\theta)+\delta_{\pi}(\theta)\big)}_{\text{meridian}}
    \qquad\text{and}\qquad
    \underbrace{\frac{1}{2\pi}\,\mathrm{d}\theta \,\times\, \delta_{0}(\varphi)}_{\text{equator}}.
\end{equation}

Keep the same probabilistic model $(\mathbb{S}^2,\nu)$ but change the metric to
\begin{equation}\label{eq: non rotation invariant distance}
\widetilde{d}\big((\theta_1,\varphi_1),(\theta_2,\varphi_2)\big)
= \sqrt{\big((\varphi_1-\varphi_2)\bmod \pi\big)^2 + \big((\theta_1-\theta_2)\bmod 2\pi\big)^2}\,.
\end{equation} 
Then the MaxEnt posterior on the equator $\{\varphi=0\}$ is
\begin{align}\label{eq: posterior on equator}
    \frac{1}{2\pi}\,\mathrm{d}\theta \,\times\, \delta_{0}(\varphi),
\end{align}
whereas the MaxEnt posterior on the meridian $\{\theta=0\}\cup\{\theta=\pi\}$ is
\begin{align}\label{eq: posterior on meridian}
    \frac{\cos\varphi}{4}\,\mathrm{d}\varphi \,\times\, \big(\delta_{0}(\theta)+\delta_{\pi}(\theta)\big).
\end{align}
\end{prop}

\begin{proof}
The argument parallels Corollary~\ref{theo: Bayes rule for Lesbegue continuous measures}; details are given in Appendix~\ref{subsec: proof of proposition bayesian posterior on the great circle with geodesic distance}.
\end{proof}

From the same \emph{prior} and the same \emph{data}, the statistician reports different summaries of the uncertainty:
\begin{itemize}
\item \textbf{Equatorial-band probability.} The posterior probability that the object lies within $30^\circ$ of the equator, $\{|\varphi|\le \pi/6\}$, is $\tfrac{1}{3}$ under $d_g$ but $\tfrac{1}{2}$ under $\widetilde d$.
\item \textbf{Central $50\%$ credible interval.} $[-45^\circ,45^\circ]$ under $d_g$ but $[-30^\circ,30^\circ]$ under $\widetilde d$.
\item \textbf{MAP estimate.} The geodesic posterior is flat and has no unique mode, whereas the map-projection posterior is unimodal with its mode at $\varphi=0$.
\end{itemize}

The marginalization rule can only report the non-invariance of the posterior under diffeomorphism, whereas Definition~\ref{def:general-bayes-rule} exhibits each answer as the posterior of a distinct model, leaving the modeller a choice to make on explicit grounds.

The difference is then located cleanly in the modelling choice made prior to the data analysis.
The model with \(d_g\) measures distances along the surface of the globe, while the model with \(\widetilde{d}\) reads them off a printed Mercator map.
Each answer is the unique posterior of its own metric space, so the multiplicity comes from the choice of metric rather than from any inconsistency of the update rule or any violation of probability theory.

One answer can be singled out once the measurement process is specified, the verdict coming from that process rather than from the update rule.
If the position of the object is given with an error proportional to distance measured on the Earth surface, then $d_g$ is the geometry the data were produced in and the uniform posterior \eqref{eq: uniform posterior on meridian} is the correct posterior estimate.
Had the object been located with an error proportional to distance measured on a printed chart, $\widetilde d$ would be the right model and its posterior \eqref{eq: posterior on meridian} the correct one.

The globe example compares two posteriors, but it cannot say which of them is right: no data-generating process is attached to it.
The next subsection attaches one.

\subsection{Calibration on $\mathrm{SO}(3)$}\label{subsec:methodology-robotics}

Once data are actually measured under a known geometry, unlike in the previous example, some Bayesian analysts are simply wrong, we show it numerically.
We track the orientation of an object from noisy measurements, a standard problem in robotics \cite{barfoot}.
One dimension is known exactly while a sensor returns noisy observations of the full orientation.
We apply Definition~\ref{def:general-bayes-rule} to infer a second dimension.
We observe that when Definition~\ref{def:general-bayes-rule} is applied with a metric that did not generate the measurement of the conditioning event, the posterior is not calibrated.

The orientation of a rigid body is expressed by attaching a frame to it and reporting how that frame sits relative to a fixed reference frame (Figure~\ref{fig:so3-frame}).
The change of frame is a rotation, the set of rotations is the group
\begin{equation}\label{eq:so3-definition}
    \mathrm{SO}(3)\defeq\big\{R\in\mathbb{R}^{3\times3}\ :\ R^{\!\top}R=I_3,\ \det R=1\big\},
\end{equation}
acting on $\mathbb{R}^3$ by $v\mapsto Rv$ and composing by matrix product.
This group is simply parametrized by $(\alpha,\beta,\gamma)\in (-\pi,\pi] \times (-\tfrac{\pi}{2},\tfrac{\pi}{2}) \times (-\pi,\pi]$ (see Figure~\ref{fig:so3-frame}).
The group is compact, hence a compact standard Borel space, hence the theory of this work applies.

\begin{figure}[!htbp]
\centering
\tdplotsetmaincoords{72}{200}
\begin{tikzpicture}[tdplot_main_coords,scale=2.0,line cap=round,line join=round]
  \pgfsetstrokeopacity{1}

  \fill[gray!12] (-1.5,-1.5,0) -- (1.7,-1.5,0) -- (1.7,1.7,0) -- (-1.5,1.7,0) -- cycle;
  \draw[gray!45] (-1.5,-1.5,0) -- (1.7,-1.5,0) -- (1.7,1.7,0) -- (-1.5,1.7,0) -- cycle;
  \node[gray!70,anchor=north west] at (-1.5,-1.5,0) {\scriptsize horizontal plane};

  \draw[->,gray!75] (0,0,0) -- (1.5,0,0) node[anchor=north east,black] {\scriptsize $e_1$};
  \draw[->,gray!75] (0,0,0) -- (0,1.5,0) node[anchor=north west,black] {\scriptsize $e_2$};
  \draw[->,gray!75] (0,0,0) -- (0,0,1.6) node[anchor=east,black] {\scriptsize $e_3$};

  \draw[dashed,black!70] (0,0,0) -- (1.028,0.863,0);
  \tdplotdrawarc[->,thick,black]{(0,0,0)}{0.60}{0}{40}{}{}
  \node at (0.60,0.34,0.02) {\scriptsize $\alpha$};

  \draw[->,very thick,red] (0,0,0) -- (1.028,0.863,0.775) node[anchor=south east,inner sep=2pt] {\scriptsize $b_1$};
  \draw[->,very thick,red] (0,0,0) -- (-1.157,0.687,0.770) node[anchor=south east] {\scriptsize $b_2$};
  \draw[->,very thick,red] (0,0,0) -- (0.085,-1.089,1.100) node[anchor=south west] {\scriptsize $b_3$};

  \draw[dashed,black!70] (0,0,0) -- (-1.157,0.687,0);

  \draw[dotted,black!60] (-1.157,0.687,0.770) -- (-1.157,0.687,0);
  \fill[black!60] (-1.157,0.687,0) circle (0.6pt);
  \draw[dotted,black!60] (1.028,0.863,0.775) -- (1.028,0.863,0);
  \fill[black!60] (1.028,0.863,0) circle (0.6pt);
  \draw[->,thick,black] (0.766,0.643,0) .. controls (0.815,0.684,0.215) and (0.780,0.655,0.385) .. (0.693,0.582,0.522);
  \node at (0.96,0.81,0.34) {\scriptsize $\beta$};

  \draw[dash dot,black!55] (0,0,0) -- (-0.598,0.713,0);
  \draw[->,thick,black] (-0.399,0.475,0.000) .. controls (-0.479,0.417,0.171) .. (-0.463,0.275,0.308);
  \node at (-0.72,0.63,0.26) {\scriptsize $\gamma$};
\end{tikzpicture}
\caption{\textbf{An orientation and its three angles.} A frame $(b_1,b_2,b_3)$ is attached to the body and reported relative to the fixed reference frame $(e_1,e_2,e_3)$; the change of frame is an element of $\mathrm{SO}(3)$, see \eqref{eq:so3-definition}. The yaw $\alpha$ is the direction the body points to within the horizontal plane, the pitch $\beta$ its elevation above that plane, and the roll $\gamma$ the spin around its own pointing axis $b_1$.}
\label{fig:so3-frame}
\end{figure}

\paragraph{Data.}
The orientation is drawn uniformly, $X_{\mathrm{real}}\sim\nu_{\mathrm{bi}}$, the uniform measure according to the volume form of the unique (up to scale) metric $g_{\mathrm{bi}}$ invariant under left and right translation,
\begin{equation}\label{eq:so3-haar}
    \mathrm{d}\nu_{\mathrm{bi}}=\frac{1}{8\pi^{2}}\,\cos\beta\;\mathrm{d}\alpha\,\mathrm{d}\beta\,\mathrm{d}\gamma\,.
\end{equation}
Two instruments then observe it.
A sensor returns the full orientation with a noise isotropic in $g_{\mathrm{bi}}$,
\begin{equation}\label{eq:so3-sensor}
    X \in \mathrm{SO}(3)  \defeq X_{\mathrm{real}}\exp(\sigma\,\eta^{\times}),\qquad \eta\sim\mathcal{N}(0,I_3),
\end{equation}
where $\eta^{\times}$ is the skew-symmetric matrix acting as $\eta^{\times}v=\eta\times v$, so that $\exp(\sigma\,\eta^{\times})$ is the rotation of axis $\eta/\lVert\eta\rVert$ and angle $\sigma\lVert\eta\rVert$, and $\sigma>0$ is the noise scale.
A second instrument accepts the trial only if $d\big(X_{\mathrm{real}},\{\alpha=\alpha^\star\}\big)<\delta$, with a fixed tolerance $\delta\ll\sigma$.
The geometry $d$ of this second instrument differs according to the experiment.
In experiment~(a) the distance is measured in $g_{\mathrm{bi}}$, as the angle between the pointing axis $b_1$ and the vertical half-plane containing the direction $\alpha^\star$ (Figure~\ref{fig:so3-frame}; an instrument locating the direction of $b_1$ in space, such as a camera, errs by an angle in space).
In experiment~(b) it is measured in the flat metric $g_{\mathrm{flat}}=\mathrm{d}\alpha^2+\mathrm{d}\beta^2+\mathrm{d}\gamma^2$, obtained by declaring the chart Euclidean, as $|\alpha-\alpha^\star|$ (an instrument reading the angle $\alpha$ itself, such as an encoder on the vertical axis, errs by a difference of numerical readings of $\alpha$).
The prior $\nu_{\mathrm{bi}}$ and the sensor \eqref{eq:so3-sensor} are the same in both experiments.

\paragraph{Task.} Assume  $\alpha = \alpha^\star$, and infer a posterior $\pi(\beta \mid X )$ on the pitch \(\beta\).
We report if $\beta_{\mathrm{real}}$ is in the centered $90\%$ interval of $\pi(\beta \mid X )$.
We repeat this experiment $2000$ times on each set up.
\footnote{The code reproducing the experiment and Figure~\ref{fig:so3-estimation} is available at \url{https://github.com/Raphael-Tresor/so3_borel_kolmogorov_experiment}; implementation details are given in Appendix~\ref{app:so3-details}.}

\paragraph{Protocol.}
Starting from no information, the prior is the uniform measure $\nu_{\mathrm{bi}}$ of \eqref{eq:so3-haar}.
The event $\{\alpha=\alpha^\star\}$ is conditioned on by Definition~\ref{def:general-bayes-rule}, which gives the posterior $\nu_{\mathrm{bi}}\big(\beta,\gamma\mid\{\alpha=\alpha^\star\}\big)$ once a metric $g\in\{g_{\mathrm{bi}},g_{\mathrm{flat}}\}$ is chosen.
From the observed $X$, the likelihood $(\beta,\gamma)\mapsto p\big(X\mid X_{\mathrm{real}}(\alpha^\star,\beta,\gamma)\big)$ is computed from \eqref{eq:so3-sensor} and combined with the posterior on the surface,
\begin{equation}\label{eq:so3-posterior}
    \pi(\beta \mid X )\ \propto\ \int_{-\pi}^{\pi} p\big(X\mid\beta,\gamma\big)\,\mathrm{d}\nu_{\mathrm{bi}}\big(\beta,\gamma\mid\{\alpha=\alpha^\star\}\big).
\end{equation}
Each experiment is analysed twice, once with each metric.
The two analyses share the prior, the likelihood and the data, and differ only by the metric in Definition~\ref{def:general-bayes-rule}, that is, by the geometry assumed for the measurement of $\alpha$.

\paragraph{Result.}
The Bayesian analysis that does not use the measurement metric is uncalibrated, see Figure~\ref{fig:so3-estimation}.
Neither metric is universally correct. 
We conclude that when the measurement geometry is known, it is the right geometry to use in Definition~\ref{def:general-bayes-rule} and the other can be biased.
The cost of ignoring that geometry depends on the data.
The two analysts agree in the small-noise limit, but the disagreement grows with the noise.
At small scale the spaces are approximately flat (the $\cos\beta$ factor is nearly constant), the geometries are approximately equivalent.
So, metric differences matter only at a scale where curvature does.

\begin{figure}[!htbp]
\centering
\includegraphics[width=\linewidth]{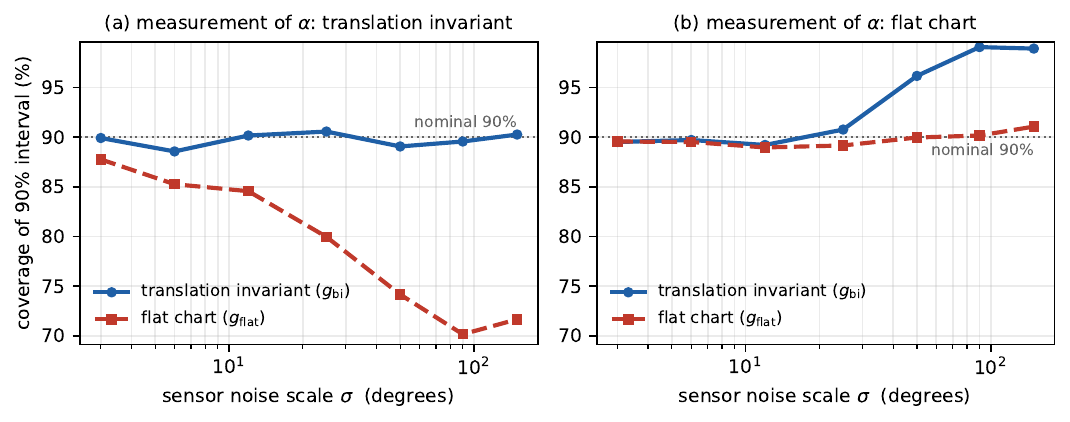}
\caption{Empirical coverage of the nominal $90\%$ posterior credible interval on the pitch $\beta$, over $2000$ trials per noise scale $\sigma$, performed for each of the two analysis metrics under each of the two metrics for measuring \(\alpha\). The blue line reports the calibration for the translation invariant metric $g_{\mathrm{bi}}$ and the red line for the flat chart metric $g_{\mathrm{flat}}$. Panel (a) is measured under $g_{\mathrm{bi}}$ and panel (b) under $g_{\mathrm{flat}}$; in each, the analyst whose metric matches that of the measurement is calibrated.}
\label{fig:so3-estimation}
\end{figure}

\subsection{Real data example: identification of the earth's magnetic field}\label{sec:real-data-example}

In the previous example, the data generating geometry was known which is rarely the case in practice.
We turn to an example where the data analysis and the geometry it presupposes are already the subject of a published disagreement.

Palaeomagnetism reconstructs the direction of the earth's magnetic field in the past from the magnetization frozen into rocks: a lava flow records the field at the time it cooled.
The instrument measures the magnetization of a rock as a vector of $\mathbb{R}^3$, and keeping only its direction gives a point in $\mathbb{S}^2$.
At each \emph{site}, several rocks that cooled at about the same time are measured.
Since the directions measured at a site scatter, a statistical analysis is needed to recover the single ancient direction of the field \citep{tauxe-paleo-2003}.
A rock that acquired a second magnetization later carries two overlapping components, and progressive demagnetization moves its measured direction along the plane spanned by them, so the sought direction is known to lie on the intersection of a plane through the origin with $\mathbb{S}^2$: a great circle \citep{MCFADDEN-1988}.

The disagreement in palaeomagnetism revolves around the way measured directions scatter around the site's true direction.
The classical treatment models that scatter by the isotropic Fisher distribution \citep{MCFADDEN-1988}, which amounts to an \emph{ambient} isotropic noise following the curvature of the sphere of directions.
An alternative treatment models it by an anisotropic Kent distribution \citep{kent-fisher-1982}, on the grounds that the scatter observed at a site is statistically anisotropic \citep{gallo-2017}, which amounts to an \emph{estimated} anisotropic metric capturing a more complex noise and the directional character of the underlying magnetism.
A third possibility is to ignore geometry altogether and apply the non-geometrical formula \eqref{eq: Baye's formula for lebesgue}, which amounts to a \emph{naive} flat metric.

At each site of the dataset provided by \citet{tauxe-paleo-2003}, the variable of interest is the point $X\in\mathbb{S}^2$ which is the normalized direction of the ancient field as recorded by a rock carrying two overlapping magnetizations.
Knowing that $X$ is a point within the great circle $A=\{x\in\mathbb{S}^2:\langle n,x\rangle=0\}$, with a pole $n = (0, 0, 1)$, translates into conditioning on the null set $A$.
The measurements of each site are randomly split into \emph{train} and \emph{test} sets to perform model selection.
The three models share the Fisher prior $\propto e^{\hat\kappa\langle\hat m,x\rangle}\,\mathrm{d}\nu(x)$, with $\nu$ the uniform measure \eqref{eq:sphere-uniform-measure} and the mean $\hat m$ and concentration $\hat\kappa$ fitted to the \emph{train} set.
They differ only in their metric: $d_g$ of Section~\ref{subsec:geometrical-resolution} (\emph{ambient}), $\widetilde d$ of \eqref{eq: non rotation invariant distance} (\emph{naive}), and the Mahalanobis Riemannian metric $g_x(u,u)=u^\top\hat\Sigma^{-1}u$, $u\in T_x\mathbb{S}^2$, with $\hat\Sigma$ the Kent estimate from the \emph{train} set (\emph{estimated}).
They are then scored by how well the resulting posterior predicts the \emph{test} set.\footnote{The code reproducing the experiment and Table~\ref{tab:realdata-scores} is available at \url{https://github.com/Raphael-Tresor/paleomagnetism-experiment-BK-paradox-paper}; implementation details are given in Appendix~\ref{sec:realdata-details}.}
The analysis is repeated at each of the $11$ sites that carry both rocks with two overlapping magnetizations and enough measurements.
Table~\ref{tab:realdata-scores} reports the scores.

\begin{table}[!htbp]
\centering

\caption{Difference in negative log predictive density between a metric and its baseline, over $76$ posterior predictions at $11$ sites: \emph{Difference} is its average per site, then over sites, and \emph{Range} gives the smallest and largest site averages. \emph{Won} counts the sites favouring the metric over its baseline. The upper block averages the two Riemannian metrics into one score; the lower blocks split the sites by how scattered their measured directions are.}
\label{tab:realdata-scores}
\begin{tabular}{lcccc}
\toprule
Metric & Baseline & Difference & Range & Won \\
\midrule
\multicolumn{5}{l}{\emph{All sites} ($11$ sites)}\\
Riemannian metrics & naive & $-0.023$ & $[-0.246,+0.009]$ & $7/11$ \\
\midrule
\multicolumn{5}{l}{\emph{Scattered sites} ($\kappa<20$, $2$ sites)}\\
estimated & ambient & $-0.238$ & $[-0.410,-0.065]$ & $2/2$ \\
\addlinespace
\multicolumn{5}{l}{\emph{Concentrated sites} ($\kappa\ge20$, $9$ sites)}\\
estimated & ambient & $+0.003$ & $[-0.004,+0.022]$ & $4/9$ \\
\bottomrule
\end{tabular}
\end{table}

On this example, declaring a Riemannian geometry and applying Definition~\ref{def:general-bayes-rule} does slightly better than ignoring geometry to apply the Bayes formula \eqref{eq: Baye's formula for lebesgue}, at $7$ of the $11$ sites, as the upper part of Table~\ref{tab:realdata-scores} shows.
The metric can moreover be tuned against Definition~\ref{def:general-bayes-rule} to improve on the metric choice. 
The tuning turns out to be worthwhile only where the data are scattered, as the lower part shows.
This dependence on the scattering is consistent with the noise argument of Section~\ref{subsec:methodology-robotics}, though this dataset is too small to turn the observation into a general statement.

\subsection{The recipe, tested on the introductory example}\label{subsec:method-vs-others}

We return to the Gaussian example of Section~\ref{sec:intro}.
The strategies of \citet{berti_notion_2020} and the construction of \citet{tjur-constructive-1975} both yield a unique posterior once a mathematical object is fixed (a strategy, respectively a map $t$), and each of \eqref{eq: intro posterior difference} and \eqref{eq: intro posterior ratio} follows from an admissible choice (Section~\ref{subsec:condi-BK-paradox}).
Neither framework says how this choice should be made.
\citet[Sec.~4.2]{bungert_lion_2022} fill this gap with a selection rule: their \emph{canonical formula}, fixed by the Euclidean geometry of $\mathbb{R}^2$, returns \eqref{eq: intro posterior difference}.
When the measured quantity is the ratio $Y=X_1/X_2$, however, they argue that \eqref{eq: intro posterior ratio} is the correct answer by declaring $(X_1,Y)$ rather than $(X_1,X_2)$ the \emph{fundamental coordinates} of the probability space, on the grounds that an instrument measuring $Y$ only reports an interval $(\hat y-\varepsilon,\hat y+\varepsilon)$, set by its measurement accuracy.
Their answer thus depends on which coordinates are declared \emph{fundamental}, a choice argued from the measurement process \cite[Prop.~1]{bungert_lion_2022} but not fixed by the canonical formula.

Our reading follows instead the geometric explanation of \citet[Sec.~3]{michael_proschan_expect_1998}.
Conditioning on $\{|X_1-X_2|\le \sqrt{2}\,\eta\}=A^{\eta}$ approximates $A$ by an enlargement of the form \eqref{eq:neighboor-extension-def} for the Euclidean metric in the chart $(x_1,x_2)$, while conditioning on $\{|X_1/X_2-1|\le\eta\}=A^{\prime,\eta}$ approximates $A^\prime=\{y=1\}$ for the same Euclidean metric in another chart $(x_1,y)$.
Transporting $A^{\prime,\eta}$ from the chart $(x_1,y)$ into the chart $(x_1,x_2)$ by an isometry does not give $A^{\eta}$ \cite[Fig.~1]{michael_proschan_expect_1998}.
We go further: the same isometry transports the prior on $(x_1,y)$ into a prior on $(x_1,x_2)$ that differs from the two independent centred Gaussians.
Read in the common chart $(x_1,x_2)$, the two posteriors therefore result from processing the same data under two models, which differ in the prior and in the assumed measurement metric.
Proposition~\ref{prop:invariance-by-homeomorphism} and Corollary~\ref{theo: Bayes rule for Lesbegue continuous measures} show that our method agrees with this interpretation.

The two approaches thus differ in the question put to the practitioner.
The frameworks above ask questions about the mathematical object conditioned upon: whether it is a set, an observed random variable or a data point; which coordinates are \emph{fundamental}; which strategy or disintegration should be used.
Our method asks a modelling question instead, a prior and a metric chosen together as one model before any data are processed.
The choice of metric translates what is measured and with which instrument: the first endows the data with an intrinsic physical geometry, the second carries the geometry of the noise proper to the instrument.
The question is standard even when the data have no strict physical meaning, as for text documents, which are given a structure through a chosen metric \citep{salton_vector_1975,kusner_word_2015}.
Definition~\ref{def:general-bayes-rule} then applies directly, without further mathematical technicalities, making the method usable by non-mathematicians.
Table~\ref{tab:comparison} summarizes the comparison.

The recipe reduces the choice of a posterior on a null set to the choice of a metric; it does not eliminate that choice.
Where the geometry is known, the model is unambiguous.
Where it is contested, models with different geometries can be tested and scored within our Bayesian framework, as is standard in statistics (Section~\ref{sec:real-data-example}).

\begin{table}[!htbp]
\centering
\caption{Constructions of a conditional probability on a null event $A$. \emph{Added structure} is what must be fixed beyond $(\mathbb{P},A)$; \emph{uniqueness up to added structure} states whether the conditional probability is unique once the added structure is fixed; \emph{invariance} lists the transformations of the model that commute with conditioning, and what must be transported with them; \emph{numerical approximation} lists the numerical methods available in the literature.}
\label{tab:comparison}
\footnotesize
\begin{tabular}{@{}>{\raggedright\arraybackslash}p{0.17\linewidth}>{\raggedright\arraybackslash}p{0.15\linewidth}>{\raggedright\arraybackslash}p{0.10\linewidth}>{\raggedright\arraybackslash}p{0.17\linewidth}>{\raggedright\arraybackslash}p{0.23\linewidth}@{}}
\toprule
Construction & Added structure & Unique up to added struct. & Invariance & Numerical approximation \\
\midrule
Regular conditional probability \citep{faden-condi-proba} & sub-$\sigma$-algebra $\mathcal{G}$ & up to $\mathbb{P}$-null sets & measurable isomorphisms, with $\mathcal{G}$ transported & none \\
\addlinespace
Strategy \citep{berti_notion_2020} & $\mathcal{G}$ and a strategy & yes & measurable isomorphisms, with $\mathcal{G}$ transported & none proposed \\
\addlinespace
Enlarged sets \citep{Renyi-1955}, \citep[Sec.~15.7]{jaynes_probability_2003} & none (any sequence $B_n\downarrow A$ in \eqref{def: conditional probability}) & no, the limit depends on $(B_n)$ & none & sampling or quadrature on $\mathbb{P}(\cdot\mid B_n)$ \\
\addlinespace
\citet{mosegaard2002probabilistic} & metric $d$ defining the enlarged sets $A^\eta$ & yes & measure-preserving isometries & as above, on $\mathbb{P}(\cdot\mid A^{\eta})$ (cf.~\eqref{eq:neighboor-extension-def}) \\
\addlinespace
\citet{tjur-constructive-1975} & map $t$ defining the enlarged sets $t^{-1}(B_n)$, and topologies & yes & homeomorphisms of the domain and codomain of $t$ & as above, on $\mathbb{P}(\cdot\mid t^{-1}(B_n))$, $B_n$ in the $2^{-n}$-ball around $\hat y$ \citep[Lemma~2.7]{ACKERMAN_FREER_ROY_2017} \\
\addlinespace
\citet{bungert_lion_2022} & Euclidean metric in \emph{fundamental} coordinates & yes & measure-preserving Euclidean isometries & quadrature on $A$ through a map satisfying their Th.~1 \\
\midrule
MaxEnt posterior (Definition~\ref{def:general-bayes-rule}) & metric $d$ and MaxEnt & yes & measure-preserving isometries (Proposition~\ref{prop:invariance-by-homeomorphism}) & as for enlarged sets, plus gradient-based MCMC and KL minimization tools (Section~\ref{sec:computation}) \\
\bottomrule
\end{tabular}
\end{table}

\section{Discussion}\label{sec: discussion}

Previous discussions of the sphere example (Section~\ref{subsec:geometrical-resolution}) rested mainly on two intuitive arguments: that the result should follow the proportionality intuition behind formulas \eqref{def: conditional probability} and \eqref{eq: Baye's formula for lebesgue} \cite[Sec.~15.7]{jaynes_probability_2003}, and that the result should be invariant under rotation, since the formulation of the initial problem appears to be.
Both intuitions are respected here, each in its own place between modelling and data analysis.

The proportionality intuition of \eqref{def: conditional probability} carries over to null sets without distortion whenever the geometry used to extend \eqref{def: conditional probability} by a limit process measures the volume of a set proportionally to the base measure of the probabilistic model.
This holds for the pair (\emph{Euclidean geometry}, \emph{Lebesgue measure}), as Corollary~\ref{theo: Bayes rule for Lesbegue continuous measures} confirms.
Figure~\ref{fig:mercator-projection} makes the intuition visual: a measure-preserving isometry carries $(\mathbb{S}^2,\widetilde d)$ onto a Euclidean rectangle.
On Figure~\ref{fig:Projection}, the Euclidean representation of the prior model, the posteriors \eqref{eq: posterior on equator}--\eqref{eq: posterior on meridian} can be read off from the grey intensity, as the proportionality intuition expects.

The rotational argument, in contrast, holds only when the formulation of the problem is explicitly invariant under rotation, which is the case for $d_g$, and when the data analysis accounts for that geometry, which is the case for Definition~\ref{def:general-bayes-rule}.
Equation~\eqref{eq: uniform posterior on meridian} of Proposition~\ref{theo: Bayesian posterior on the great circle with geodesic distance} establishes it, and confirms the prediction of \cite[Sec.~7.1]{gyenis_conditioning_2017}.

Beyond the sphere, the MaxEnt derivation also closes a gap at the foundations of Bayesian inference.
The classical ratio formula~\eqref{def: conditional probability} is compatible with Cox-type axioms for plausible reasoning \cite{cox_algebra_2001,knuth_inductive_2002,vanhornConstructingLogicPlausible2003}, yet the continuous-variable update~\eqref{eq: Baye's formula for lebesgue} does not follow from those axioms despite being used in the same spirit.
Deriving both from a single MaxEnt principle, and showing they coincide wherever the classical rule applies (Section~\ref{subsec:recovering-classical-formulas}), supplies the justification the axioms leave open.

\makeatletter
\define@key{x sphericalkeys}{radius}{\def\myradius{#1}}
\define@key{x sphericalkeys}{theta}{\def\mytheta{#1}}
\define@key{x sphericalkeys}{phi}{\def\myphi{#1}}
\tikzdeclarecoordinatesystem{x spherical}{%
    \setkeys{x sphericalkeys}{#1}%
    \pgfpointxyz{\myradius*cos(\mytheta)}{\myradius*sin(\mytheta)*cos(\myphi)}{\myradius*sin(\mytheta)*sin(\myphi)}}

\define@key{y sphericalkeys}{radius}{\def\myradius{#1}}
\define@key{y sphericalkeys}{theta}{\def\mytheta{#1}}
\define@key{y sphericalkeys}{phi}{\def\myphi{#1}}
\tikzdeclarecoordinatesystem{y spherical}{%
    \setkeys{y sphericalkeys}{#1}%
    \pgfpointxyz{\myradius*sin(\mytheta)*cos(\myphi)}{\myradius*cos(\mytheta)}{\myradius*sin(\mytheta)*sin(\myphi)}}

\define@key{z sphericalkeys}{radius}{\def\myradius{#1}}
\define@key{z sphericalkeys}{theta}{\def\mytheta{#1}}
\define@key{z sphericalkeys}{phi}{\def\myphi{#1}}
\tikzdeclarecoordinatesystem{z spherical}{%
    \setkeys{z sphericalkeys}{#1}%
    \pgfmathsetmacro{\Xtest}{sin(\tdplotmaintheta)*cos(\tdplotmainphi-90)*sin(\mytheta)*cos(\myphi)
    +sin(\tdplotmaintheta)*sin(\tdplotmainphi-90)*sin(\mytheta)*sin(\myphi)
    +cos(\tdplotmaintheta)*cos(\mytheta)}
    \pgfmathsetmacro{\ntest}{ifthenelse(\Xtest<0,0,1)}
    \ifnum\ntest=0 \xdef\MCheatOpa{0.3}\else\xdef\MCheatOpa{1}\fi
    \pgfpointxyz{\myradius*sin(\mytheta)*cos(\myphi)}{\myradius*sin(\mytheta)*sin(\myphi)}{\myradius*cos(\mytheta)}}

\pgfdeclareplothandler{\pgfplothandlercurveto}{}{%
  point macro=\pgf@plot@curveto@handler@initial,
  jump macro=\pgf@plot@smooth@next@moveto,
  end macro=\pgf@plot@curveto@handler@finish
}
\def\pgf@plot@smooth@next@moveto{%
  \pgf@plot@curveto@handler@finish%
  \global\pgf@plot@startedfalse%
  \global\let\pgf@plotstreampoint\pgf@plot@curveto@handler@initial%
}
\def\pgf@plot@curveto@handler@initial#1{%
  \pgf@process{#1}%
  \pgf@xa=\pgf@x\pgf@ya=\pgf@y%
  \pgf@plot@first@action{\pgfqpoint{\pgf@xa}{\pgf@ya}}%
  \xdef\pgf@plot@curveto@first{\noexpand\pgfqpoint{\the\pgf@xa}{\the\pgf@ya}}%
  \global\let\pgf@plot@curveto@first@support=\pgf@plot@curveto@first%
  \global\let\pgf@plotstreampoint=\pgf@plot@curveto@handler@second%
}
\def\pgf@plot@curveto@handler@second#1{%
  \pgf@process{#1}%
  \xdef\pgf@plot@curveto@second{\noexpand\pgfqpoint{\the\pgf@x}{\the\pgf@y}}%
  \global\let\pgf@plotstreampoint=\pgf@plot@curveto@handler@third%
  \global\pgf@plot@startedtrue%
}
\def\pgf@plot@curveto@handler@third#1{%
  \pgf@process{#1}%
  \xdef\pgf@plot@curveto@current{\noexpand\pgfqpoint{\the\pgf@x}{\the\pgf@y}}%
  \pgf@xa=\pgf@x \pgf@ya=\pgf@y%
  \pgf@process{\pgf@plot@curveto@first}\advance\pgf@xa by-\pgf@x \advance\pgf@ya by-\pgf@y
  \pgf@xa=\pgf@plottension\pgf@xa \pgf@ya=\pgf@plottension\pgf@ya
  \pgf@process{\pgf@plot@curveto@second}\pgf@xb=\pgf@x \pgf@yb=\pgf@y \pgf@xc=\pgf@x \pgf@yc=\pgf@y
  \advance\pgf@xb by-\pgf@xa \advance\pgf@yb by-\pgf@ya \advance\pgf@xc by\pgf@xa \advance\pgf@yc by\pgf@ya
  \@ifundefined{MCheatOpa}{}{%
    \pgf@plotstreamspecial{\pgfsetstrokeopacity{\MCheatOpa}}}%
  \edef\pgf@marshal{\noexpand\pgfsetstrokeopacity{\noexpand\MCheatOpa}
    \noexpand\pgfpathcurveto{\noexpand\pgf@plot@curveto@first@support}{\noexpand\pgfqpoint{\the\pgf@xb}{\the\pgf@yb}}{\noexpand\pgf@plot@curveto@second}
    \noexpand\pgfusepathqstroke
    \noexpand\pgfpathmoveto{\noexpand\pgf@plot@curveto@second}}%
  {\pgf@marshal}%
  \global\let\pgf@plot@curveto@first=\pgf@plot@curveto@second%
  \global\let\pgf@plot@curveto@second=\pgf@plot@curveto@current%
  \xdef\pgf@plot@curveto@first@support{\noexpand\pgfqpoint{\the\pgf@xc}{\the\pgf@yc}}%
}
\def\pgf@plot@curveto@handler@finish{%
  \ifpgf@plot@started%
    \pgfpathcurveto{\pgf@plot@curveto@first@support}{\pgf@plot@curveto@second}{\pgf@plot@curveto@second}%
  \fi%
}

\pgfmathsetmacro{\RadiusSphere}{3}
\tdplotsetmaincoords{60}{131}
\pgfmathtruncatemacro{\Levels}{5}
\pgfmathtruncatemacro{\NLong}{9}
\pgfmathsetmacro{\dTheta}{180/(\Levels-1)}  
\pgfmathsetmacro{\dPhi}{360/(\NLong)}       

\begin{figure}[!htbp]
\centering
\begin{subfigure}{0.5\linewidth}
\raggedright
\begin{minipage}[t]{\linewidth}
            
\begin{tikzpicture}[tdplot_main_coords,line cap=round,line join=round,  draw=red]
    \pgfmathsetmacro{\RadiusSphere}{3}
    \tdplotsetmaincoords{60}{131}
  \pgfmathtruncatemacro{\Levels}{5}
  \pgfmathtruncatemacro{\NLong}{9}
  \pgfmathsetmacro{\dTheta}{180/(\Levels-1)}  
  \pgfmathsetmacro{\dPhi}{360/(\NLong)}       

  \draw[tdplot_screen_coords,ball color=gray,fill opacity=0.45] (0,0,0) circle (\RadiusSphere);

  \path (z spherical cs:radius=\RadiusSphere,theta=0,phi=0)
    node[circle,fill=black,inner sep=1.0pt,opacity=\MCheatOpa] (P-N) {};
  \path (z spherical cs:radius=\RadiusSphere,theta=180,phi=0)
    node[circle,fill=black,inner sep=1.0pt,opacity=\MCheatOpa] (P-S) {};

  \foreach \r in {1,...,\numexpr\Levels-2\relax}{
    \pgfmathsetmacro{\Theta}{\r*\dTheta} 
    \foreach \l in {0,...,\numexpr\NLong-1\relax}{
      \pgfmathsetmacro{\Phi}{\l*\dPhi}
      \path (z spherical cs:radius=\RadiusSphere,theta=\Theta,phi=\Phi)
        node[circle,fill=black,inner sep=1.0pt,opacity=\MCheatOpa] (P-\r-\l) {};
    }
  }

  \foreach \l in {0,...,\numexpr\NLong-1\relax}{
    \pgfmathsetmacro{\Phi}{\l*\dPhi}
    \draw[very thin]
      plot[variable=\t,domain=0:180,samples=90,smooth]
      (z spherical cs:radius=\RadiusSphere,theta=\t,phi=\Phi);
  }

  \foreach \r in {1,...,\numexpr\Levels-2\relax}{
    \pgfmathsetmacro{\Theta}{\r*\dTheta}
    \draw[very thin]
      plot[variable=\t,domain=0:360,samples=181,smooth]
      (z spherical cs:radius=\RadiusSphere,theta=\Theta,phi=\t);
  }

  \foreach \r in {1,2}{%
    \pgfmathsetmacro{\ThA}{\r*\dTheta}
    \pgfmathsetmacro{\ThB}{(\r+1)*\dTheta}
    \foreach \l in {0,...,\numexpr\NLong-1\relax}{%
      \draw[very thin]
        plot[variable=\t,domain=0:1,samples=30,smooth]
        (z spherical cs:
          radius=\RadiusSphere,
          theta={(1-\t)*\ThA + \t*\ThB},
          phi={(1-\t)*((\l+1)*\dPhi) + \t*(\l*\dPhi)});
    }%
  }
\end{tikzpicture}
    \caption{Uniform prior \eqref{eq:sphere-uniform-measure} on $\mathbb{S}^2$.}
    \label{fig:sphere parametrisation}
    \end{minipage}
    \end{subfigure}
    \hfill
    \begin{subfigure}{0.45\linewidth}
        \raggedleft
        \begin{minipage}[t]{\linewidth}
            \includegraphics[scale=0.4,height=4.3cm]{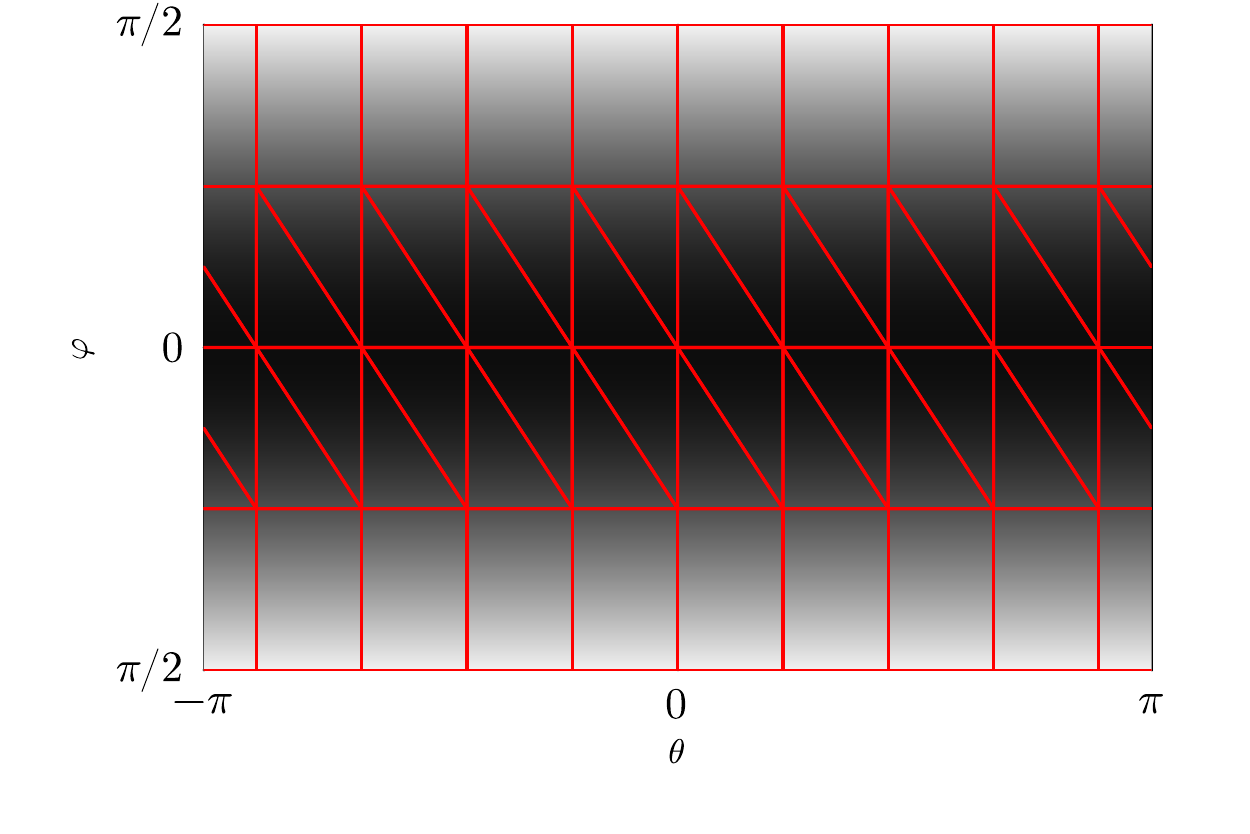}
            \caption{Representation of $(\mathbb{S}^2,\widetilde d)$ via an isometry to the rectangle $(-\tfrac{\pi}{2},\tfrac{\pi}{2})\times(-\pi,\pi]$. The push-forward of \eqref{eq:sphere-uniform-measure} is given by \eqref{eq: uniform measure on sphere}.}
            \label{fig:Projection}
        \end{minipage}
    \end{subfigure}
    \caption{\textbf{Measure-preserving chart and metric change.} The chart \eqref{eq:chart on sphere} maps $\mathbb{S}^2$ to an angular rectangle. When $\mathbb{S}^2$ is equipped with an alternative metric $\widetilde d$ (defined in \eqref{eq: non rotation invariant distance}), the chart becomes isometric to the Euclidean metric on the rectangle. The transformation illustrates how area and distance distort and why geometric information is required for coherent conditioning; see Section~\ref{sec: discussion}.    Although the original measure on \(\mathbb{S}^2\) is uniform  (Figure~\ref{fig:sphere parametrisation}), the push-forward measure on the Euclidean space (Figure~\ref{fig:Projection}), which preserves the geometric structure, is not uniform.}
    \label{fig:mercator-projection}
\end{figure}

\section*{Acknowledgments}
The authors would like to thank Bert de Vries and Thijs van de Laar for valuable discussions.
This publication is part of the project ``ROBUST: Trustworthy AI-based Systems for Sustainable Growth'' with project number KICH3.LTP.20.006, which is (partly) financed by the Dutch Research Council (NWO), GN Hearing, and the Dutch Ministry of Economic Affairs and Climate Policy (EZK) under the program LTP KIC 2020-2023.

\section*{Code Availability}
Python code reproducing the $\mathrm{SO}(3)$ state-estimation experiment is available at \url{https://github.com/Raphael-Tresor/so3_borel_kolmogorov_experiment}, and code reproducing the palaeomagnetic experiment at \url{https://github.com/Raphael-Tresor/paleomagnetism-experiment-BK-paradox-paper}.

\bibliography{ref}

\appendix
\counterwithin{theo}{section}

\section{Conditional Probability}\label{sec: conditional proba}

This appendix revisits classical facts about conditional probability and presents modernized, self‑contained proofs that align with the MaxEnt viewpoint adopted in the paper. We first derive the standard conditional probability on a finite $\sigma$‑algebra as the unique solution of a MaxEnt problem, and then extend to general measurable spaces.

Let $A,B$ be events. The $\sigma$‑algebra $\sigma(\{A,B\})$ generated by $\{A,B\}$ has the four atoms
\begin{equation}\label{eq: elementaries for A and B}
\{A\cap B,\ A\cap B^{\complement},\ A^{\complement}\cap B,\ A^{\complement}\cap B^{\complement}\}.
\end{equation}
Let $\mathbb{P}$ be a prior probability on $\sigma(\{A,B\})$. After observing $A$, we seek a posterior $\mathbb{Q}$ on $\sigma(\{A,B\})$. The classical solution is the conditional probability $\mathbb{P}(\cdot\mid A)$ in \eqref{def: conditional probability}. Alternatively, we may compute $\mathbb{Q}$ as the minimizer of the KL divergence relative to $\mathbb{P}$, subject to $\mathbb{Q}(A)=1$:
\begin{subequations}\label{eq: ME optimisation problem for original Baye's formula}
\begin{align}
&\inf \ \Ent{\mathbb{P}}{\mathbb{Q}},\\
&\text{subject to }\ \mathbb{Q}\ \text{probability on }\sigma(\{A,B\}), \label{eq: simple case condition  1}\\
&\qquad\qquad\ \ \ \mathbb{Q}(A)=1. \label{eq: simple case condition  2}
\end{align}
\end{subequations}

\begin{theo}\label{theo: deriving conditional proba from MaxEnt}
Assume $\mathbb{P}(A)>0$. The unique solution to \eqref{eq: ME optimisation problem for original Baye's formula} is $\mathbb{P}(\cdot\mid A)$.
\end{theo}

\begin{proof}
Write $p_i=\mathbb{P}(E_i)$ and $q_i=\mathbb{Q}(E_i)$ on the atoms $E_1=A\cap B$, $E_2=A\cap B^{\complement}$, $E_3=A^{\complement}\cap B$, $E_4=A^{\complement}\cap B^{\complement}$. Then
\begin{equation}\label{eq: Shanon rel-ent}
\Ent{\mathbb{P}}{\mathbb{Q}}=\sum_{i=1}^4 q_i\log\frac{q_i}{p_i}.
\end{equation}
Constraints \eqref{eq: simple case condition  1}–\eqref{eq: simple case condition  2} imply $q_3=q_4=0$ and $q_1+q_2=1$. Setting $t=q_1\in[0,1]$, we minimize
\[
f(t)=t\log\frac{t}{p_1}+(1-t)\log\frac{1-t}{p_2}.
\]
The function $f$ is strictly convex on $(0,1)$; $f'(t)=\log p_1-\log t-\log p_2+\log(1-t)$ vanishes at $t^\star=\frac{p_1}{p_1+p_2}=\frac{\mathbb{P}(A\cap B)}{\mathbb{P}(A)}$. Hence $q_1^\star=t^\star$, $q_2^\star=1-t^\star$, yielding
\[
\mathbb{Q}^\star(B)=\frac{\mathbb{P}(A\cap B)}{\mathbb{P}(A)}=\mathbb{P}(B\mid A).
\]
Thus $\mathbb{Q}^\star=\mathbb{P}(\cdot\mid A)$.
\end{proof}

\begin{rem}
The argument extends verbatim to finite partitions: MaxEnt with the single hard constraint “mass one on the observed atom” returns the normalized restriction of the prior to that atom.
\end{rem}

We now record the corresponding statement on general measurable spaces.

\begin{prop}\label{theo:condi-proba-from-NaxEnt-positive-case}
Let $(E,\mathcal{B}(E))$ be a standard Borel space with probability measure $\nu$, and let $A\in\mathcal{B}(E)$ satisfy $\nu(A)>0$. Then
\begin{equation}\label{eq: MaxEnt gives conditional proba}
\argmin_{\substack{\mu(E)=1\\ \mu(A)=1}} \Ent{\nu}{\mu}
= \frac{\mathds{1}_{A}}{\nu(A)}\,\nu.
\end{equation}
\end{prop}

\begin{proof}
Any feasible $\mu$ must satisfy $\mu\ll\nu$ (otherwise $\Ent{\nu}{\mu}=+\infty$), and $\mu(E\setminus A)=0$. Thus $\mu$ has density $f\ge 0$ supported on $A$ with $\int_A f\,\mathrm{d}\nu=1$, i.e., $\mathrm{d}\mu=f\,\mathrm{d}\nu$. Compute
\[
\Ent{\nu}{\mu}=\int_A f\log f\,\mathrm{d}\nu
= \int_A f\log\frac{f}{\nu(A)^{-1}}\,\mathrm{d}\nu + \log\!\big(\nu(A)^{-1}\big)\int_A f\,\mathrm{d}\nu.
\]
Hence
\[
\Ent{\nu}{\mu}
= \Ent{\tfrac{\mathds{1}_{A}}{\nu(A)}\nu}{\mu} + \log\!\big(\nu(A)^{-1}\big)
\ \ge\ \log\!\big(\nu(A)^{-1}\big)
= \Ent{\nu}{\tfrac{\mathds{1}_{A}}{\nu(A)}\nu},
\]
using nonnegativity of KL (Proposition~\ref{prop: rel ent is positive}). Equality holds iff $\mu=\frac{\mathds{1}_{A}}{\nu(A)}\nu$.
\end{proof}


\subsection{Proof of Lemma~\ref{lem:distance contraint equiv to support inclusion even for limit}}\label{subsec: proof of lemma Support inclusion from truncated constraints}

\begin{proof}
Fix $R>0$. For any open $U\subset A^{\complement}$ we have $\inf_{x\in U} d_R(x,A)>0$, hence
\[
\mu_n(U)\le \frac{1}{\inf_{x\in U} d_R(x,A)^2}\int_E d_R(x,A)^2\,\mathrm{d}\mu_n(x)\xrightarrow[n\to\infty]{}0.
\]
By the Portmanteau theorem, $\mu(U)=0$ for all such $U$, and outer regularity yields $\mu(B)=0$ for all Borel $B\subset A^{\complement}$.
\end{proof}

\section{Convergences to delta measure}\label{app: proof of gaussian over nu convergence to delta measure}

This appendix collects weak convergence results used in the derivation of the MaxEnt posteriors in the main text.
The explicit solutions of the MaxEnt problems rest on the convex duality of \cite{borwein_partially_1992,borwein_duality_1991}. 

\begin{theo}[\cite{borwein_duality_1991}, Cor.~2.6 and Thm.~4.8]\label{theo: dual problem of ME}
Let $(E,\mathcal{E},\nu)$ be a measure space with $\sigma$-finite $\nu$. Let $a_1,\ldots,a_n\in L^\infty_\nu(E)$ and $b\in\mathbb{R}^n$. Consider
\[
\text{(ME)}\quad
\left\{
\begin{array}{ll}
\displaystyle \inf & \displaystyle \int_E f(x)\,\lnr{f(x)}\,\mathrm{d}\nu(x) \\[1ex]
\text{subject to} & \displaystyle \int_E a_i(x) f(x)\,\mathrm{d}\nu(x)=b_i,\quad i=1,\ldots,n,\\
& f\ge 0,\quad f\in L^1_\nu.
\end{array}
\right.
\]
and the dual
\[
\text{(DME)}\quad
\left\{
\begin{array}{ll}
\displaystyle \max & \displaystyle \lambda^{\top} b
- \int_E \expr{\,-1+\sum_{i=1}^n \lambda_i a_i(x)\,}\,\mathrm{d}\nu(x) \\[1ex]
\text{subject to} & \lambda\in\mathbb{R}^n.
\end{array}
\right.
\]
Assume there exists a feasible $f$ with $f>0$ $\nu$-a.e. and $\int f\ln f\,\mathrm{d}\nu<\infty$ (Slater-type condition). Then:
\begin{enumerate}
\item Strong duality holds: the infimum of (ME) equals the maximum of (DME), and both are attained.
\item The unique minimizer is
\begin{equation}\label{eq: ME solution form}
f^\star(x)=\expr{ -1 + \sum_{i=1}^n \lambda^\star_i\, a_i(x)},
\end{equation}
where $\lambda^\star$ is a maximizer of (DME).
\end{enumerate}
\end{theo}

\begin{proof}
Apply \cite[Cor.~2.6]{borwein_duality_1991} to $(E,\mathcal{E},\nu)$ with the strictly convex integrand $x\mapsto x\ln x$ on $[0,\infty)$. The representation \eqref{eq: ME solution form} follows from \cite[Thm.~4.8]{borwein_duality_1991} and the superlinear growth $\lim_{x\to\infty}(x\ln x)/x=\infty$.
\end{proof}

\begin{lem}\label{lem:outside-of-A-measure-collapse}
Consider the setting of Definition~\ref{def:general-bayes-rule}. Fix $R>0$ and $0<\varepsilon<R$. Then
\begin{equation}\label{eq:goal-to-prove-1000}
\mu_{a,R}\big((A^{\varepsilon})^{\complement}\big)\xrightarrow[a\to\infty]{}0,
\end{equation}
where $\mathrm{d}\mu_{a,R}(x)\propto \exp\{-a\,d_R(x,A)^2\}\,\mathrm{d}\nu(x)$ and $A^\varepsilon=\{x:\ d(x,A)\le \varepsilon\}$.
\end{lem}

\begin{proof}
For $0<\varepsilon<R$,
\begin{align*}
\frac{\int_{(A^{\varepsilon})^{\complement}} e^{-a d_R(x,A)^2}\,\mathrm{d}\nu(x)}
     {\int_{A^{\varepsilon}} e^{-a d_R(x,A)^2}\,\mathrm{d}\nu(x)}
&= \frac{\int_{(A^{\varepsilon})^{\complement}} e^{-a d_R(x,A)^2}\,\mathrm{d}\nu(x)}
         {\int_{A^{\varepsilon}\setminus A^{\varepsilon/2}} e^{-a d_R(x,A)^2}\,\mathrm{d}\nu(x)+\int_{A^{\varepsilon/2}} e^{-a d_R(x,A)^2}\,\mathrm{d}\nu(x)}\\
&\le \frac{\nu\!\big((A^{\varepsilon})^{\complement}\big)\,e^{-a\varepsilon^2}}
           {\nu\!\big(A^{\varepsilon}\setminus A^{\varepsilon/2}\big)\,e^{-a\varepsilon^2}
            +\nu\!\big(A^{\varepsilon/2}\big)\,e^{-a(\varepsilon/2)^2}}\\
&= \frac{\nu\!\big((A^{\varepsilon})^{\complement}\big)}
          {\nu\!\big(A^{\varepsilon}\setminus A^{\varepsilon/2}\big)
           +\nu\!\big(A^{\varepsilon/2}\big)\,e^{a\varepsilon^2(1-1/4)}}\xrightarrow[a\to\infty]{}0.
\end{align*}
Thus $\int_{(A^{\varepsilon})^{\complement}} e^{-a d_R^2}\,\mathrm{d}\nu$ is negligible compared to $\int_{A^\varepsilon} e^{-a d_R^2}\,\mathrm{d}\nu$, proving \eqref{eq:goal-to-prove-1000}.
\end{proof}

\begin{lem}\label{lem: exponentiated distance convergence to delta measure}
Let $(F,d)$ be a $\sigma$-compact standard Borel space with Radon measure $\mathcal{N}$. Fix $R>0$ and let $p\in L^1_{\mathcal{N}}(F)$ be continuous at $\hat y$. Suppose $\hat y\in \mathrm{int}(\operatorname{supp}(\mathcal{N}))$ and $y\mapsto e^{-a d_R(y,\hat y)^2}\in L^1_{\mathcal{N}}(F)$ for $a>0$. Then
\begin{equation}\label{eq:expo-normal-probability-measure_2}
\int_F \frac{e^{-a d_R(y,\hat y)^2}}{\int_F e^{-a d_R(u,\hat y)^2}\,\mathrm{d}\mathcal{N}(u)}\,p(y)\,\mathrm{d}\mathcal{N}(y)
\ \xrightarrow[a\to\infty]{}\ p(\hat y).
\end{equation}
\end{lem}

\begin{proof}
Let $\varepsilon>0$. By continuity of $p$ at $\hat y$, choose $0<r<R$ such that $|p(y)-p(\hat y)|\le \varepsilon$ for all $y\in B(\hat y,r)=\{y:\ d(y,\hat y)\le r\}$. Then
\begin{align}\notag
&\Bigg|\int_{B(\hat y,r)} p(y)\,e^{-a d_R(y,\hat y)^2}\,\mathrm{d}\mathcal{N}(y)
 - p(\hat y)\int_{B(\hat y,r)} e^{-a d_R(y,\hat y)^2}\,\mathrm{d}\mathcal{N}(y)\Bigg|\\\label{eq:upper-bound-1000}
&\qquad \le \varepsilon \int_{B(\hat y,r)} e^{-a d_R(y,\hat y)^2}\,\mathrm{d}\mathcal{N}(y).
\end{align}
Outside $B(\hat y,r)$, exponential decay yields
\begin{subequations}\label{eq:negligeable far from y hat}
\begin{align}
&\frac{\int_{B(\hat y,r)^{\complement}} p(y)\,e^{-a d_R(y,\hat y)^2}\,\mathrm{d}\mathcal{N}(y)}
     {\int_{F} e^{-a d_R(u,\hat y)^2}\,\mathrm{d}\mathcal{N}(u)}\\
&\qquad\le \frac{e^{-a r^2}\int_{B(\hat y,r)^{\complement}} |p(y)|\,\mathrm{d}\mathcal{N}(y)}
           {\int_{B(\hat y,r/2)} e^{-a d_R(u,\hat y)^2}\,\mathrm{d}\mathcal{N}(u)
            +\int_{B(\hat y,r)\setminus B(\hat y,r/2)} e^{-a d_R(u,\hat y)^2}\,\mathrm{d}\mathcal{N}(u)}\\
&\qquad\le \frac{\int_{B(\hat y,r)^{\complement}} |p(y)|\,\mathrm{d}\mathcal{N}(y)}
           {e^{a r^2(1-1/4)}\,\mathcal{N}\!\big(B(\hat y,r/2)\big)
            + \mathcal{N}\!\big(B(\hat y,r)\setminus B(\hat y,r/2)\big)}
\ \xrightarrow[a\to\infty]{}0.\label{eq:far-negligible}
\end{align}
\end{subequations}
A similar bound without $p$ gives
\begin{equation}\label{eq:denom-localization}
\frac{\int_{B(\hat y,r)^{\complement}} e^{-a d_R(y,\hat y)^2}\,\mathrm{d}\mathcal{N}(y)}
     {\int_{F} e^{-a d_R(u,\hat y)^2}\,\mathrm{d}\mathcal{N}(u)}
\ \xrightarrow[a\to\infty]{}0,
\end{equation}
hence
\begin{equation}\label{eq:equiv_1000}
\int_{B(\hat y,r)} e^{-a d_R(u,\hat y)^2}\,\mathrm{d}\mathcal{N}(u)
\ \sim_{a\to\infty}\
\int_{F} e^{-a d_R(u,\hat y)^2}\,\mathrm{d}\mathcal{N}(u).
\end{equation}
Combining \eqref{eq:upper-bound-1000}–\eqref{eq:equiv_1000} and dividing by the common normalizer yields
\[
\left|\frac{\int_{F} p(y)\,e^{-a d_R(y,\hat y)^2}\,\mathrm{d}\mathcal{N}(y)}
           {\int_{F} e^{-a d_R(u,\hat y)^2}\,\mathrm{d}\mathcal{N}(u)}
      - p(\hat y)\right|
\ \le\ \varepsilon + o_{a\to\infty}(1),
\]
and since $\varepsilon>0$ is arbitrary, the limit equals $p(\hat y)$.
\end{proof}

\subsection{Proof of Theorem~\ref{theo:general bayes rule has a coherent solution}}\label{subsec: proof of theorem Pre-Posterior}

\begin{proof}
The function 
\[a \mapsto \int_E d_R(x,A)^2 \d{}\mu_a(x) \] 
converges to \(0\) when \(a\to\infty\)  by Lemma~\ref{lem:outside-of-A-measure-collapse}. 
This implies the feasibility problem~\eqref{prob: Bayes for Hausdorff space} for small $\sigma$. 
By Theorem~\ref{theo: dual problem of ME}, Problem~\eqref{prob: Bayes for Hausdorff space} is then  \textit{well-defined} and has a unique solution of the form \eqref{def: pre Bayes posterior}.
Fix a sequence $\sigma_k\downarrow 0$ and write $a_k\coloneqq a(\sigma_k)$.
Every subsequence of $(a_k)$ has a further subsequence that is either bounded, tends to $-\infty$, or tends to $+\infty$; we rule out the first two cases.
If $\nu(E\setminus A)>0$ and a subsequence of $(a_k)$ is bounded, then $\int d_R^2\,\mathrm{d}\mu_{a_k}\not\to 0$ along it, contradicting the hypothesis.
By outer regularity of \(\nu\) and hypothesis $\nu(E\setminus A)>0$, for an \(\eta<R\), \(\nu(A^{\eta,\complement})>0\) (see \eqref{eq:neighboor-extension-def}).
It remains to exclude a subsequence $a_k\to-\infty$; suppose one exists, so that $a_k<0$ along it.
Since $-a_k>0$, the map $d^2\mapsto e^{-a_k d^2}$ is increasing, hence bounded above by $e^{-a_k\eta^2/4}$ on $A^{\eta/2}$, below by $e^{-a_k\eta^2/4}$ on $A^{\eta}\setminus A^{\eta/2}$, and below by $e^{-a_k\eta^2}$ on $A^{\eta,\complement}$, giving

\begin{align}\notag
    \frac{\mu_{a_k}(A^{\frac{\eta}{2}})}{\mu_{a_k}\left(\left(A^{\eta/2}\right)^\complement\right)}& = \frac{\int_{A^{\eta/2}} e^{-a_k d^2(x,A)}\d{\nu(x)}}{\int_{\left(A^{\eta/2}\right)^\complement} e^{-a_k d^2(x,A)} \d{\nu(x)}}\\\notag
    & = \frac{\int_{A^{\eta/2}} e^{-a_k d^2(x,A)}\d{\nu(x)}}{\int_{A^{\eta}\setminus A^{\eta/2}} e^{-a_k d^2(x,A)} \d{\nu(x)} + \int_{A^{\eta,\complement}} e^{-a_k d^2(x,A)} \d{\nu(x)}}\\\notag
    & \leq   \frac{e^{-a_k \frac{\eta^2}{4}}\nu(A^{\eta/2})}{e^{-a_k\frac{\eta^2}{4}} \nu(A^{\eta}\setminus A^{\eta/2}) + e^{-a_k\eta^2} \nu(A^{\eta,\complement})  } \\\notag
    &= \frac{\nu(A^{\eta/2})}{\nu(A^{\eta}\setminus A^{\eta/2}) + e^{-a_k\frac{3\eta^2}{4}} \nu(A^{\eta,\complement})} \xrightarrow[a_k\to -\infty]{} 0 \,,
\end{align}

so \(\mu_{a_k}(A^{\eta/2})\to 0\), equivalently \(\mu_{a_k}\left(\left(A^{\eta/2}\right)^\complement\right)\to 1\), and the inequality
\begin{align}\notag
    \E{X\sim \mu_{a_k}}{d^2(X,A)} \geq  \frac{\eta^2}{4}  \mu_{a_k}\left(\left(A^{\frac{\eta}{2}}\right)^\complement\right)
\end{align}
contradicts the hypothesis.
Having excluded the bounded and the $-\infty$ cases, every subsequence of $(a_k)$ has a further subsequence tending to $+\infty$; hence $a_k\to+\infty$, and as $\sigma_k\downarrow 0$ was arbitrary, $a(\sigma)\uparrow\infty$.
\end{proof}

\subsection{Proof of Proposition~\ref{prop:general bayes rule indepedant of R clamping}}\label{subsec: proof of proposition Independence of truncation}

\begin{proof}
By Theorem~\ref{theo:general bayes rule has a coherent solution}, the solutions to \eqref{prob: Bayes for Hausdorff space} with truncations $R$ and $R'$ are
\[
\mathrm{d}\mu_{a,R}(x)=\frac{e^{-a\,d_R(x,A)^2}}{\int_E e^{-a\,d_R(u,A)^2}\,\mathrm{d}\nu(u)}\,\mathrm{d}\nu(x),\qquad
\mathrm{d}\mu_{a,R'}(x)=\frac{e^{-a\,d_{R'}(x,A)^2}}{\int_E e^{-a\,d_{R'}(u,A)^2}\,\mathrm{d}\nu(u)}\,\mathrm{d}\nu(x),
\]
and the MaxEnt posterior $\mu^\star$ is the weak limit of $\mu_{a,R}$ as $a\to\infty$. It suffices to show $\mu_{a,R'}\Rightarrow \mu^\star$ as $a\to\infty$.
Let
\begin{equation}\label{eq:104}
C_{a,r}\coloneqq \int_E e^{-a\,d_r(x,A)^2}\,\mathrm{d}\nu(x),\quad r>0.
\end{equation}
Fix $\varepsilon>0$. By exponential concentration near $A$ (Lemma~\ref{lem:outside-of-A-measure-collapse}), for any $r>0$,
\begin{equation}\label{eq:equivalence around the support}
C_{a,r} \sim_{a\to\infty} \int_{A^\varepsilon} e^{-a\,d_r(x,A)^2}\,\mathrm{d}\nu(x).
\end{equation}
On $A^\varepsilon$ (see \eqref{eq:neighboor-extension-def}), $d_r(x,A)=d_{r'}(x,A)$ whenever $r,r'>\varepsilon$, hence from \eqref{eq:equivalence around the support},
\[
C_{a,R}\sim_{a\to\infty} \int_{A^\varepsilon} e^{-a\,d_R^2}\,\mathrm{d}\nu
= \int_{A^\varepsilon} e^{-a\,d_{R'}^2}\,\mathrm{d}\nu \sim_{a\to\infty} C_{a,R'}.
\]
Therefore $C_{a,R}/C_{a,R'}\to 1$ as $a\to\infty$.
Let $f\in\mathcal{C}_b(E)$. Decompose
\[
\int_E f\,\mathrm{d}\mu_{a,r}
= \frac{1}{C_{a,r}}\left(\int_{A^\varepsilon} f\,e^{-a\,d_r^2}\,\mathrm{d}\nu + \int_{(A^\varepsilon)^{\complement}} f\,e^{-a\,d_r^2}\,\mathrm{d}\nu\right).
\]
By boundedness of $f$ and Lemma~\ref{lem:outside-of-A-measure-collapse}, the second integral is negligible compared to $C_{a,r}$. Using $d_R=d_{R'}$ on $A^\varepsilon$ and $C_{a,R}\sim C_{a,R'}$,
\[
\int_E f\,\mathrm{d}\mu_{a,R}
\sim_{a\to\infty} \frac{\int_{A^\varepsilon} f\,e^{-a\,d_R^2}\,\mathrm{d}\nu}{\int_{A^\varepsilon} e^{-a\,d_R^2}\,\mathrm{d}\nu}
= \frac{\int_{A^\varepsilon} f\,e^{-a\,d_{R'}^2}\,\mathrm{d}\nu}{\int_{A^\varepsilon} e^{-a\,d_{R'}^2}\,\mathrm{d}\nu}
\sim_{a\to\infty} \int_E f\,\mathrm{d}\mu_{a,R'}.
\]
Taking limits along $a\to\infty$ and using the hypothesis $\mu_{a,R}\Rightarrow \mu^\star$, we conclude $\mu_{a,R'}\Rightarrow \mu^\star$. Hence the MaxEnt posterior is independent of the truncation parameter.
\end{proof}

%
%

\subsection{Proof of Proposition~\ref{prop:invariance-by-homeomorphism}}\label{subsec: proof of proposition Invariance under measure-preserving isometries}

\begin{proof}
Let $h:(E,d)\to(\widetilde{E},\widetilde{d})$ be an isometry and let $h_{\#}\nu$ denote the push-forward of $\nu$. Fix a closed set $A\subset \operatorname{supp}(h_{\#}\nu)$ for which the $(\nu,h^{-1}(A))$ MaxEnt posterior (Definition~\ref{def:general-bayes-rule}) is well-defined on $(E,d)$. Denote by $\mu_a$ and $\widetilde{\mu}_a$ the corresponding pre‑posteriors given by Theorem~\ref{theo:general bayes rule has a coherent solution} on $(E,d)$ and $(\widetilde{E},\widetilde{d})$, respectively.

Let $f\in\mathcal{C}_b(\widetilde{E})$. We show
\[
\lim_{a\to\infty}\, \mathbb{E}_{\widetilde{\mu}_a}[f]
=\lim_{a\to\infty}\, \mathbb{E}_{\mu_a}[f\circ h].
\]

Since $h$ is an isometry, for all $x\in E$,
\begin{equation}\label{eq:105}
e^{-a\,d(x,h^{-1}(A))^2}=e^{-a\,\widetilde{d}(h(x),A)^2}.
\end{equation}
Hence, for the normalizing constants
\begin{align}
\int_E e^{-a\,d(x,h^{-1}(A))^2}\,\mathrm{d}\nu(x)
&= \int_E e^{-a\,\widetilde{d}(h(x),A)^2}\,\mathrm{d}\nu(x)
= \int_{\widetilde{E}} e^{-a\,\widetilde{d}(y,A)^2}\,\mathrm{d}(h_{\#}\nu)(y) \label{eq: equivalence of pre-MaxEntPosterior constant}
\end{align}
by change of variables under $h$ and the definition of $h_{\#}\nu$.

Split $f=f_+-f_-$ with $f_\pm\ge 0$ bounded. Using \eqref{eq:105} and the push-forward identity,
\begin{align}
\int_{\widetilde{E}} f_+(y)\,e^{-a\,\widetilde{d}(y,A)^2}\,\mathrm{d}(h_{\#}\nu)(y)
= \int_E f_+\big(h(x)\big)\,e^{-a\,d(x,h^{-1}(A))^2}\,\mathrm{d}\nu(x),\label{eq:106}
\end{align}
and similarly,
\begin{align}
\int_{\widetilde{E}} f_-(y)\,e^{-a\,\widetilde{d}(y,A)^2}\,\mathrm{d}(h_{\#}\nu)(y)
= \int_E f_-\big(h(x)\big)\,e^{-a\,d(x,h^{-1}(A))^2}\,\mathrm{d}\nu(x).\label{eq:107}
\end{align}
Subtracting \eqref{eq:107} from \eqref{eq:106} yields
\begin{equation}\label{eq:108}
\int_{\widetilde{E}} f(y)\,e^{-a\,\widetilde{d}(y,A)^2}\,\mathrm{d}(h_{\#}\nu)(y)
= \int_E f\big(h(x)\big)\,e^{-a\,d(x,h^{-1}(A))^2}\,\mathrm{d}\nu(x).
\end{equation}

Using the pre‑posterior representation \eqref{def: pre Bayes posterior} for $\mu_a$ and $\widetilde{\mu}_a$, and combining \eqref{eq: equivalence of pre-MaxEntPosterior constant} with \eqref{eq:108}, we obtain for all $a>0$,
\[
\mathbb{E}_{\widetilde{\mu}_a}[f]
= \mathbb{E}_{\mu_a}[f\circ h].
\]
Passing to the limit $a\to\infty$ completes the proof:
\[
\lim_{a\to\infty}\mathbb{E}_{\widetilde{\mu}_a}[f]
= \lim_{a\to\infty}\mathbb{E}_{\mu_a}[f\circ h].
\]
\end{proof}

\subsection{Proof of Theorem~\ref{theo: criteria of existence}}\label{subsec: proof of theorem criteria of existence}

\begin{proof}
Throughout, $0<\eta\le\eta_0$ with $\eta<R$, and $A^\eta=\{x\in E:\ d(x,A)\le\eta\}$ is the set $C^\eta$ for $C=A$.
Since $\nu(A)=0$, Theorem~\ref{theo:general bayes rule has a coherent solution} gives $a(\sigma)\uparrow\infty$ as $\sigma\downarrow 0$, so it suffices to prove $\mu_a\Rightarrow\mu^\star$ as $a\to\infty$, with $\mu^\star(C)\coloneqq f_C(0)/f_A(0)$.

We first prove that $\mu^\star$ is a probability measure on $A$.
For each $\eta$, the map $C\mapsto\nu(C^\eta)=\nu\big(A^\eta\cap\pi^{-1}(C)\big)$ is a finite Borel measure on $A$.
By continuity of $f_C$ at $0$,
\begin{equation*}
f_C(0)=\lim_{\eta\downarrow 0}\frac{1}{\eta}\int_0^\eta f_C(t)\,\mathrm{d}t=\lim_{\eta\downarrow 0}\frac{\nu(C^\eta)}{\eta}\,.
\end{equation*}
Along any sequence $\eta_k\downarrow 0$, the measures $C\mapsto\nu(C^{\eta_k})/\eta_k$ thus converge setwise to $C\mapsto f_C(0)$, which is therefore countably additive by Nikodym's convergence theorem \cite[Th.~4.6.3(i)]{bogachev_spaces_2007}.
Since $f_A(0)>0$, $\mu^\star$ is a Borel probability measure on $A$, viewed as a measure on $E$ supported on $A$.

Now, we prove that 
\begin{equation}\label{eq:sub_obj_1}
    \mu_a(C^\eta)\xrightarrow[a\to\infty]{}\mu^\star(C)  \,,\;\forall  C\subset A \text{Borel}\,.
\end{equation}
The push-forward of $\nu$ by $x\mapsto d(x,A)$ restricted to $C^\eta$ has density $f_C$ on $[0,\eta]$, and $d_R=d$ on $A^\eta$, hence
\begin{equation*}
\mu_a(C^\eta)=\frac{\int_0^\eta e^{-at^2}f_C(t)\,\mathrm{d}t}{Z_a}\,,\qquad Z_a=\int_0^\eta e^{-at^2}f_A(t)\,\mathrm{d}t+\int_{(A^\eta)^\complement}e^{-a\,d_R(x,A)^2}\,\mathrm{d}\nu(x)\,.
\end{equation*}
Let $\delta\in(0,\eta)$ be such that $f_C$ is bounded on $[0,\delta]$.
The substitution $t=s/\sqrt{a}$ and dominated convergence give $\sqrt{a}\int_0^\delta e^{-at^2}f_C(t)\,\mathrm{d}t\to\frac{\sqrt{\pi}}{2}f_C(0)$, while $\sqrt{a}\int_\delta^\eta e^{-at^2}f_C(t)\,\mathrm{d}t\le\sqrt{a}\,e^{-a\delta^2}\nu(C^\eta)\to 0$.
The same holds for $f_A$, and the second term of $Z_a$ is at most $e^{-a\eta^2}$ since $d_R(x,A)\ge\eta$ outside $A^\eta$.
Hence $\sqrt{a}Z_a\to\frac{\sqrt{\pi}}{2}f_A(0)>0$ and $\mu_a(C^\eta)\to f_C(0)/f_A(0)=\mu^\star(C)$.

Finally, we prove that  $\mu_a\Rightarrow\mu^\star$.
Let $g:E\to\mathbb{R}$ be bounded and $L$-Lipschitz, and let $\lambda_a\coloneqq\pi_{\#}\big(\mu_a|_{A^\eta}\big)$, a measure on $A$ with $\lambda_a(C)=\mu_a(C^\eta)$.

By definition of $\lambda_a$,
\begin{equation*}
\int_E g\,\mathrm{d}\mu_a-\int_A g\,\mathrm{d}\lambda_a=\int_{A^\eta}\big(g(x)-g(\pi(x))\big)\,\mathrm{d}\mu_a(x)+\int_{(A^\eta)^\complement}g\,\mathrm{d}\mu_a\,.
\end{equation*}
Since $|g(x)-g(\pi(x))|\le L\,d(x,\pi(x))=L\,d(x,A)\le L\eta$ on $A^\eta$,
\begin{equation*}
\Big|\int_E g\,\mathrm{d}\mu_a-\int_A g\,\mathrm{d}\lambda_a\Big|\le L\eta+\|g\|_\infty\,\mu_a\big((A^\eta)^\complement\big)\,,
\end{equation*}
and the last term tends to $0$ by Lemma~\ref{lem:outside-of-A-measure-collapse}.
Choose a simple function $s=\sum_{i=1}^{m}c_i\mathds{1}_{C_i}$ on $A$, with $C_i\subset A$ Borel, such that $\sup_A|g-s|\le\eta$.
Result of \eqref{eq:sub_obj_1} gives $\int_A s\,\mathrm{d}\lambda_a\to\int_A s\,\mathrm{d}\mu^\star$, and $\lambda_a(A),\mu^\star(A)\le 1$, so
\begin{equation*}
\limsup_{a\to\infty}\Big|\int_E g\,\mathrm{d}\mu_a-\int_E g\,\mathrm{d}\mu^\star\Big|\le (L+2)\,\eta\,.
\end{equation*}
Letting $\eta\downarrow 0$, $\int_E g\,\mathrm{d}\mu_a\to\int_E g\,\mathrm{d}\mu^\star$ for every bounded Lipschitz $g$, which implies $\mu_a\Rightarrow\mu^\star$ by the Portmanteau theorem.
Hence the MaxEnt posterior exists and equals $\mu^\star$, which is \eqref{eq:existence-criterion-posterior}.
\end{proof}

\subsection{Proof of Theorem~\ref{theo: Bayesian map is an extension of conditional probability}}\label{subsec: proof of theorem bayesian map is an extension of conditional probability}

\begin{proof}
Write $\mu_n(\mathrm{d}x)=C_{a_n}^{-1} e^{-a_n d_R(x,A)^2}\,\mathrm{d}\nu(x)$ with $a_n\to\infty$,\newline and $C_{a_n}=\int_E e^{-a_n d_R(x,A)^2}\,\mathrm{d}\nu(x)$. For any $r>0$,
\begin{align*}
|C_{a_n}-\nu(A)|&= \int_{A^{r,\complement}} e^{-a_nd_R(x,A)^2} d\nu(x) + \int_{A^r/A} e^{-a_nd_R(x,A)^2} d\nu(x) \\
&\le \nu(A^{r,\complement})e^{-a_n r^2}+\nu(A^r\setminus A)\xrightarrow[n\to\infty]{}\nu(A^r\setminus A)\xrightarrow[r\downarrow 0]{}0
\end{align*}
by outer regularity. For bounded continuous $f$, dominated convergence yields
\[
C_{a_n}^{-1}\int_E f(x)e^{-a_n d_R(x,A)^2}\,\mathrm{d}\nu(x)\to \frac{1}{\nu(A)}\int_A f(x)\,\mathrm{d}\nu(x).
\]
\end{proof}

\subsection{Proof of Proposition~\ref{lem: convergence criteria to delta measure Polish space}}\label{subsec: proof of lemma convergence criteria to delta measure Polish space}

\begin{proof}
For bounded continuous $f:\mathcal{Y}\to\mathbb{R}$ and $0<r<R$,
\[
\big|\mathbb{E}_{\mu_n}[f(Y)]-f(\hat y)\big|
\le r + 2\|f\|_\infty\,\mu_n\big(B_{\mathcal{Y}}^{\complement}(\hat y,r)\big),
\]
and $\mu_n\big(B_{\mathcal{Y}}^{\complement}(\hat y,r)\big)\,r^2\le \mathbb{E}_{\mu_n}[d_{\mathcal{Y},R}(Y,\hat y)^2]\le \mathbb{E}_{\mu_n}[d_R((X,Y),A)^2]\to 0$ as $n\to\infty$, whence the claim.
\end{proof}

\subsection{Proof of Theorem~\ref{theo: Bayes rule for product space}}\label{subsec: proof of theorem Bayes rule for product space}
\begin{proof}
By \eqref{eq:condition on the metric},$(E\times F,d)$ is a standard Borel space. 
We consider $A=E\times\{\hat y\}$ and fix $R>0$. For the sequence \eqref{prob: Bayes for Hausdorff space} with $\sigma_n\downarrow 0$, Theorem~\ref{theo:general bayes rule has a coherent solution} gives the pre‑posterior
\begin{equation}\label{eq: random variable observation pre posterior}
\mathrm{d}\mu_a(x,y)
= C_a^{-1}\exp\!\Big\{-\frac{a}{2}\,d_{F,R}(y,\hat y)^2\Big\}\,\mathrm{d}\nu(x,y)\,\end{equation}
with
\begin{equation*}
C_a=\int_{E\times F}\exp\!\Big\{-\frac{a}{2}\,d_{F,R}(y,\hat y)^2\Big\}\,\mathrm{d}\nu(x,y),
\end{equation*}
and $a\to\infty$ as $\sigma\downarrow 0$.

By Tonelli's theorem, the evidence $\widetilde p$ lies in $L^1(\mathcal{N}_F)$, and by assumption it is continuous at $\hat y$ with $0<\widetilde p(\hat y)<\infty$.
Then
\begin{equation}\label{eq: normalisation constant for lebesgue contininous case v2}
C_a=\int_F e^{-\frac{a}{2}d_{F,R}(y,\hat y)^2}\,\widetilde p(y)\,\mathrm{d}\mathcal{N}_F(y).
\end{equation}

The proof of Lemma~\ref{lem: exponentiated distance convergence to delta measure} uses the continuity of the integrand only at $\hat y$; applying it to $(F,d_F,\mathcal{N}_F)$ with integrand $\widetilde p$ yields
\begin{equation}\label{eq:constant-limit-behave}
\left(\int_F e^{-\frac{a}{2}d_{F,R}(y,\hat y)^2}\,\mathrm{d}\mathcal{N}_F(y)\right)^{-1}\!C_a
\ \xrightarrow[a\to\infty]{}\ \widetilde p(\hat y).
\end{equation}

For $f\in\mathcal{C}_b(E\times F)$, set $g_f(y)\coloneqq\int_E f(x,y)\,p(x,y)\,\mathrm{d}\mathcal{N}_E(x)$, so that $|g_f|\le\|f\|_\infty\,\widetilde p$ and $g_f\in L^1(\mathcal{N}_F)$.
As $y\to\hat y$, $p(\cdot,y)\to p(\cdot,\hat y)$ pointwise by continuity of $p$, and $\int_E p(\cdot,y)\,\mathrm{d}\mathcal{N}_E=\widetilde p(y)\to\widetilde p(\hat y)<\infty$, so Scheffé's lemma gives $\|p(\cdot,y)-p(\cdot,\hat y)\|_{L^1(\mathcal{N}_E)}\to 0$.
Hence
\begin{equation*}
|g_f(y)-g_f(\hat y)|\le\|f\|_\infty\,\|p(\cdot,y)-p(\cdot,\hat y)\|_{L^1(\mathcal{N}_E)}+\int_E|f(x,y)-f(x,\hat y)|\,p(x,\hat y)\,\mathrm{d}\mathcal{N}_E(x)\xrightarrow[y\to\hat y]{}0\,,
\end{equation*}
the last term by dominated convergence.
Thus $g_f$ is continuous at $\hat y$, and Lemma~\ref{lem: exponentiated distance convergence to delta measure} with integrand $g_f$ gives
\begin{align}\notag
&\left(\int_F e^{-\frac{a}{2}d_{F,R}(y,\hat y)^2}\,\mathrm{d}\mathcal{N}_F(y)\right)^{-1}
\int_{E\times F} f(x,y)\,e^{-\frac{a}{2}d_{F,R}(y,\hat y)^2} p(x,y)\,\mathrm{d}\mathcal{N}_E(x)\,\mathrm{d}\mathcal{N}_F(y)\\\label{eq: numerator limit of Bayes rule for product space}&\qquad \qquad\xrightarrow[a\to\infty]{}\ \int_E f(x,\hat y)\,p(x,\hat y)\,\mathrm{d}\mathcal{N}_E(x).
\end{align}
Combining \eqref{eq:constant-limit-behave} and \eqref{eq: numerator limit of Bayes rule for product space} and dividing, since $\widetilde p(\hat y)>0$, yields
\[
\mathbb{E}_{\mu_a}[f]\ \xrightarrow[a\to\infty]{}\ \frac{\int_E f(x,\hat y)\,p(x,\hat y)\,\mathrm{d}\mathcal{N}_E(x)}{\int_E p(x,\hat y)\,\mathrm{d}\mathcal{N}_E(x)}.
\]
Since the limit functional is exactly integration with respect to
\[
\frac{p(x,\hat y)}{\int_E p(u,\hat y)\,\mathrm{d}\mathcal{N}_E(u)}\,\mathrm{d}\mathcal{N}_E(x)\ \otimes\ \delta_{\hat y}(y),
\]
we conclude the weak convergence of \(\mu_a\).
\end{proof}

\subsection{Proof of Proposition~\ref{theo: Bayesian posterior on the great circle with geodesic distance}}\label{subsec: proof of proposition bayesian posterior on the great circle with geodesic distance}

\begin{proof}
The geodesic distance to the equator is
\begin{equation}\label{eq: geodesic distance to great circles}
d_g\big((\theta,\varphi),\{\varphi=0\}\big)=|\varphi|\,.
\end{equation}

By Theorem~\ref{theo:general bayes rule has a coherent solution}, the pre‑posterior for $\{\varphi=0\}$ is
\begin{equation*}
\mathrm{d}\mu_a(\varphi,\theta)=C_a\,e^{-a\varphi^2}\cos\varphi\,\frac{\mathrm{d}\varphi\,\mathrm{d}\theta}{4\pi},\qquad
C_a^{-1}=\int_{-\pi}^{\pi}\int_{-\frac{\pi}{2}}^{\frac{\pi}{2}}e^{-a\varphi^2}\cos\varphi\,\frac{\mathrm{d}\varphi\,\mathrm{d}\theta}{4\pi}.
\end{equation*}

Change variables $\psi=\sqrt{a}\varphi$ to obtain
\begin{align}
\sqrt{a}C_a^{-1}
&=\frac{\sqrt{a}}{2}\int_{-\frac{\pi}{2}}^{\frac{\pi}{2}}e^{-a\varphi^2}\cos\varphi\,\mathrm{d}\varphi
=\frac{1}{2}\int_{-\sqrt{a}\frac{\pi}{2}}^{\sqrt{a}\frac{\pi}{2}}e^{-\psi^2}\cos\!\Big(\frac{\psi}{\sqrt{a}}\Big)\,\mathrm{d}\psi
\xrightarrow[a\to\infty]{}\frac{1}{2}\int_{\mathbb{R}}e^{-\psi^2}\,\mathrm{d}\psi.\label{eq: limit of normalisation constant for phi equal zero great circle with geodesic distance}
\end{align}

For bounded continuous $f$, using \eqref{eq: limit of normalisation constant for phi equal zero great circle with geodesic distance} we have
\begin{align*}
\mathbb{E}_{\mu_a}[f]
&=\frac{C_a}{\sqrt{a}}\frac{1}{4\pi}\int_{-\sqrt{a}\frac{\pi}{2}}^{\sqrt{a}\frac{\pi}{2}}\Bigg(\int_{-\pi}^{\pi}f\Big(\tfrac{\psi}{\sqrt{a}},\theta\Big)\,\mathrm{d}\theta\Bigg)e^{-\psi^2}\cos\!\Big(\tfrac{\psi}{\sqrt{a}}\Big)\,\mathrm{d}\psi
\xrightarrow[a\to\infty]{}\frac{1}{2\pi}\int_{-\pi}^{\pi}f(0,\theta)\,\mathrm{d}\theta \,.
\end{align*}
 Thus the posterior on $\{\varphi=0\}$ is uniform. Any great circle is the image of the equator under a rotation $R\in SO(3)$, which is an isometry of $(\mathbb{S}^2,d_g)$ preserving the uniform measure ($R_{\#}\nu=\nu$); Proposition~\ref{prop:invariance-by-homeomorphism} then transports the uniform equatorial posterior to a uniform posterior on that great circle.

Now, consider the metric $\widetilde d$ in \eqref{eq: non rotation invariant distance}. 
For $\{\varphi=0\}$, $\widetilde d((\theta,\varphi),\{\varphi=0\})=|\varphi|$ (cf.~\eqref{eq: non rotation invariant distance}), identical to $d_g$ in \eqref{eq: geodesic distance to great circles}. By Proposition~\ref{theo: Bayesian posterior on the great circle with geodesic distance}, the posterior is uniform, proving \eqref{eq: posterior on equator}.

For $\{\theta=0\}\cup\{\theta=\pi\}$, $\widetilde d((\theta,\varphi),\{\theta=0\}\cup\{\theta=\pi\})=\min(|\theta|,|\pi-\theta|,|\pi+\theta|)$, which differs from $d_g$. Definition~\ref{def:general-bayes-rule} gives the pre‑posterior
\begin{equation}\label{eq: pr-posterior of pole to pole great circle with map distance}
\mathrm{d}\mu_a(\varphi,\theta)=
\begin{cases}
C_a\,e^{-a\theta^2}\cos\varphi\,\frac{\mathrm{d}\varphi\,\mathrm{d}\theta}{4\pi}, & \theta\in[-\tfrac{\pi}{2},\tfrac{\pi}{2}],\\[0.5ex]
C_a\,e^{-a|\pi-|\theta||^2}\cos\varphi\,\frac{\mathrm{d}\varphi\,\mathrm{d}\theta}{4\pi}, & \theta\in(-\pi,-\tfrac{\pi}{2}]\cup[\tfrac{\pi}{2},\pi].
\end{cases}
\end{equation}
As $a\to\infty$, the mass concentrates on $\theta\in\{0,\pi\}$; the same Laplace argument as above yields
\begin{equation}\label{eq: bayes posterior on medidian from non euclidian distance}
\mathrm{d}\mu^\star(\varphi,\theta)=\frac{\cos\varphi}{4}\,\mathrm{d}\varphi\times\mathds{1}_{\{0,\pi\}}(\theta),
\end{equation}
which is the cosine‑weighted posterior claimed in \eqref{eq: posterior on meridian}.
\end{proof}

\section{Details on the $\mathrm{SO}(3)$ Experiment}\label{app:so3-details}

\paragraph{Posterior on $\{\alpha=\alpha^\star\}$.}
Definition~\ref{def:general-bayes-rule} applied to the prior $\nu_{\mathrm{bi}}$ gives
\begin{equation}\label{eq:so3-two-posteriors}
    \nu_{\mathrm{bi}}\big(\mathrm{d}\beta,\mathrm{d}\gamma\mid\{\alpha=\alpha^\star\}\big)=
    \begin{cases}
        \dfrac{1}{2\pi^{2}}\,\mathrm{d}\beta\,\mathrm{d}\gamma & \text{in } g_{\mathrm{bi}}\,,\\[2mm]
        \dfrac{\cos\beta}{4\pi}\,\mathrm{d}\beta\,\mathrm{d}\gamma & \text{in } g_{\mathrm{flat}}\,.
    \end{cases}
\end{equation}

\paragraph{Likelihood.}
Every candidate $X(\beta,\gamma)=R_z(\alpha^\star)R_y(\beta)R_x(\gamma)$ lies on $\{\alpha=\alpha^\star\}$, since the posterior \eqref{eq:so3-two-posteriors} is supported on it, and is scored by its rotation angle $\theta$ to the observation,
\begin{equation}\label{eq:so3-likelihood}
    p\big(X\mid \beta,\gamma\big)\ \propto\ \exp\!\Big(-\frac{\theta\big(X(\beta,\gamma),X\big)^2}{2\sigma^2}\Big),
    \qquad
    \theta(R,R')=\arccos\frac{\operatorname{tr}(R^{\!\top}R')-1}{2}\,,
\end{equation}
for $R,R'\in\mathrm{SO}(3)$.
This is the Gaussian-in-angle approximation standard in robotics \cite{barfoot}; replacing it by the exact density of \eqref{eq:so3-sensor} changes the coverage by less than the Monte Carlo error.
Integrating \eqref{eq:so3-likelihood} against \eqref{eq:so3-two-posteriors} over $\gamma$ gives \eqref{eq:so3-posterior}.

\paragraph{Numerical method details.}
Each experiment runs $2000$ trials per noise scale \\$\sigma\in\{3,6,12,25,50,90,150\}$ degrees.
In each trial a target $\alpha^\star$ is drawn uniformly on $(-\pi,\pi]$, and the trial is kept only if $d_g\big(X_{\mathrm{real}},\{\alpha=\alpha^\star\}\big)<\delta$, with $\delta=0.29^\circ$, where $g$ is the metric of the second instrument, $g_{\mathrm{bi}}$ in experiment~(a) and $g_{\mathrm{flat}}$ in experiment~(b).
The empirical coverage of an analysis is the fraction of kept trials whose true pitch $\beta_{\mathrm{real}}$ lies in the central $90\%$ interval of $\pi(\beta\mid X)$.
The truths and the noise variables $\eta$ of \eqref{eq:so3-sensor} are drawn once and reused across all $\sigma$, and the posterior is computed from the analytical expressions \eqref{eq:so3-two-posteriors} and \eqref{eq:so3-likelihood} on a $181\times181$ grid in $(\beta,\gamma)$, not sampled.

\paragraph{Robustness to the tolerance $\delta$.}
The tolerance $\delta$ is the resolution of the instrument measuring $\alpha$; it is fixed within each run, and small against $\sigma$ so that this instrument is nearly exact next to the sensor.
To check that the results do not depend on its value, the whole calibration run is repeated with $\delta\in\{0.1^\circ,0.29^\circ,1^\circ\}$, all other settings unchanged (Figure~\ref{fig:so3-delta}).
The results do not change noticeably with $\delta$, so the conclusions of Section~\ref{subsec:methodology-robotics} do not depend on its choice.

\begin{figure}[!htbp]
\centering
\includegraphics[width=\linewidth]{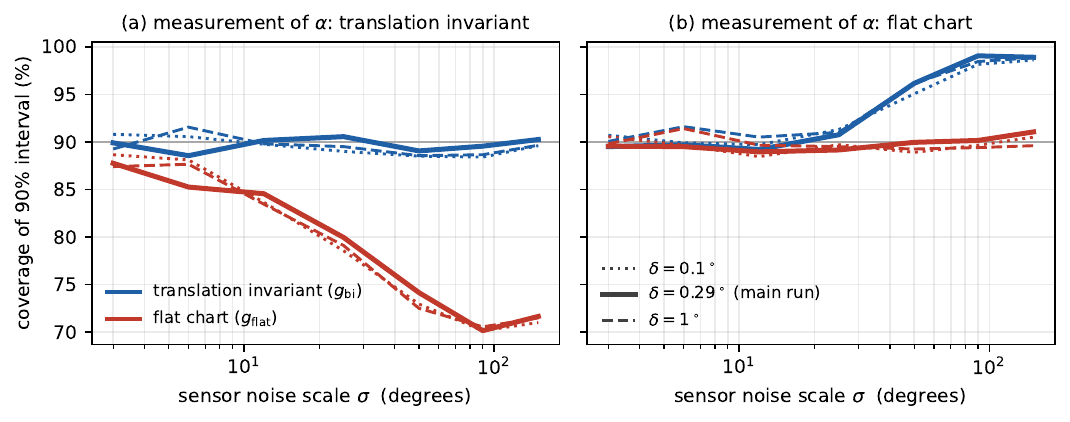}
\caption{Coverage of the nominal $90\%$ interval on $\beta$ for three tolerances $\delta$ of the instrument measuring $\alpha$, $2000$ trials per point. Colours are those of Figure~\ref{fig:so3-estimation}: blue for the analysis in $g_{\mathrm{bi}}$, red for the analysis in $g_{\mathrm{flat}}$. The line style gives $\delta$: dotted $0.1^\circ$, solid $0.29^\circ$ (the run of that figure), dashed $1^\circ$. The thin grey line is the nominal $90\%$.}
\label{fig:so3-delta}
\end{figure}

\section{Details on the Palaeomagnetic Experiment}\label{sec:realdata-details}

We preprocessed the dataset of \citet{tauxe-paleo-2003} before the experiment.
It holds $502$ measurements on rocks over $46$ sites, of which the experiment scores $76$ predictions at $11$ sites.

Two kinds of measurement matter.
When the demagnetization path of a rock follows a single component, a line is fitted to it and the measurement is a direction, a point of $\mathbb{S}^2$ ($154$ measurements); when the rock carries two overlapping magnetizations, only the plane spanned by them can be fitted, hence a great circle ($42$ measurements over $23$ sites).
The great circles are the null sets conditioned on, and the directions are what the posteriors are scored against, so a site is usable only if it supplies both.

A separation into \emph{train} and \emph{test} sets needs at least two directions per site, which leaves $14$ of the $23$ sites carrying a great circle (Table~\ref{tab:realdata-cascade}).
Three of them are further removed because their training directions nearly coincide, so that the estimated Fisher concentration $\hat\kappa$ diverges.

\begin{table}[!htbp]
\centering
\caption{Reduction of the archived dataset to the scored sample. Each condition is required by the separation into \emph{train} and \emph{test} sets of Section~\ref{sec:real-data-example}.}
\label{tab:realdata-cascade}
\begin{tabular}{lcc}
\toprule
Condition & Sites & Great circles \\
\midrule
At least one great circle                          & $23$ & $42$ \\
\quad and at least two directions                  & $14$ & $23$ \\
\quad and a finite Fisher concentration $\hat\kappa$ & $11$ & $17$ \\
\bottomrule
\end{tabular}
\end{table}

At the $11$ remaining sites, every direction forms the \emph{test} set once, the other directions of its site forming the \emph{train} set, and is scored against every great circle of its site, which gives the $76$ predictions.
A site contributes the product of its two counts, so the sites are unequally represented, and Table~\ref{tab:realdata-scores} therefore averages per site.

The \emph{estimated} metric of that table is read from the Fisher--Bingham (FB5) moment estimator of \citet{kent-fisher-1982}, fitted at each site to its \emph{train} set alone.
It returns two axes $e_1,e_2$ of the tangent plane at the mean $\hat m$ with variances $\hat\tau_1,\hat\tau_2$, and the form $u^\top\hat\Sigma^{-1}u$ of Section~\ref{sec:real-data-example} stands for $\sum_{j=1}^{2}\langle u,e_j\rangle^2/\hat\tau_j$.
Definition~\ref{def:general-bayes-rule} needs a metric on the whole sphere, and the implementation uses, at every $x\in\mathbb{S}^2$, $g_x(u,u)=\sum_{j=1}^{2}\langle u,e_j\rangle^2/\big(\hat\tau_j\langle\hat m,x\rangle^2\big)$, which coincides with the form of Section~\ref{sec:real-data-example} at $\hat m$.
The resulting posterior on the circle $A$ is
\begin{equation*}
\nu\big(\mathrm{d}x\mid A\big)\propto e^{\hat\kappa\langle\hat m,x\rangle}\Big(\sum_{j=1}^{2}\frac{\langle t(x),e_j\rangle^2}{\hat\tau_j}\Big)^{1/2}\,\mathrm{d}s(x)\,,
\end{equation*}
with $t(x)$ the unit tangent of $A$ at $x$ and $\mathrm{d}s$ the arclength of $d_g$.
At site \texttt{sv61}, each \emph{train} set holds two directions only, so that $\hat\tau_2$ vanishes; it is floored at $10^{-12}$.

\end{document}